\pdfoutput=1
\documentclass[a4paper,11pt,reqno]{amsart}

\usepackage[margin=1in,includehead,includefoot]{geometry}
\usepackage[utf8]{inputenc}
\usepackage[T1]{fontenc}
\usepackage{lmodern,microtype}
\usepackage[foot]{amsaddr}
\usepackage{upref}
\usepackage{mathtools,amssymb,amsthm}
\usepackage[mathcal]{euscript}
\usepackage{enumitem}
\usepackage{tikz}
\usepackage{collcell,array}
\usepackage{colortbl}
\usepackage[b]{esvect} 
\usepackage{graphicx}
\usepackage{hyperref}
\usepackage{bookmark}

\hypersetup{colorlinks=true,linkcolor=black,citecolor=black,filecolor=black,urlcolor=black}

\makeatletter
\def\@setfoot@addresses{%
  \def\author##1{}%
  \newif\if@firstaddr\@firstaddrtrue
  \def\address##1##2{%
    \if@firstaddr\@firstaddrfalse\else\par\fi
    \ignorespaces##1\unskip: \ignorespaces##2%
  }%
  \def\email##1##2{}%
  \addresses
}
\makeatother

\usepackage{mycommands}

\newtheorem{theorem}{Theorem}
\newtheorem{lemma}[theorem]{Lemma}
\theoremstyle{remark}
\newtheorem*{claim}{Claim}

\setlist[enumerate]{label=\textup{(\roman*)}, noitemsep, topsep=3pt plus 3pt, leftmargin=*, widest=iii}


\let\leq\leqslant
\let\geq\geqslant

\hypersetup{
  pdftitle={Sparse Kneser graphs are Hamiltonian},
  pdfauthor={Torsten M\"utze, Jerri Nummenpalo, Bartosz Walczak}
}

\title{Sparse Kneser graphs are Hamiltonian}

\author{Torsten M\"utze}
\address[Torsten M\"utze]{Department of Computer Science, University of Warwick, Coventry CV47AL, United Kingdom}
\email{torsten.mutze@warwick.ac.uk}

\author{Jerri Nummenpalo}
\address[Jerri Nummenpalo]{Department of Computer Science, ETH Z\"urich, 8092 Z\"urich, Switzerland}
\email{njerri@inf.ethz.ch}

\author{Bartosz Walczak}
\address[Bartosz Walczak]{Department of Theoretical Computer Science, Faculty of Mathematics and Computer Science, Jagiellonian University, Krak\'ow, Poland}
\email{walczak@tcs.uj.edu.pl}

\thanks{An extended abstract of this paper appeared in \emph{STOC 2018: Proceedings of the 50th Annual ACM SIGACT Symposium on Theory of Computing}, pages 912--919.\ ACM, New York, 2018.}
\thanks{Torsten M\"utze is also affiliated with Charles University, Faculty of Mathematics and Physics, and was supported by Czech Science Foundation grant GA~19-08554S, and by German Science Foundation grant~413902284.}
\thanks{Bartosz Walczak was partially supported by National Science Centre of Poland grant 2015/17/D/ST1/00585.}

\begin{document}

\begin{abstract}
For integers~$k\geq 1$ and~$n\geq 2k+1$, the \emph{Kneser graph}~$K(n,k)$ is the graph whose vertices are the $k$-element subsets of~$\{1,\ldots,n\}$ and whose edges connect pairs of subsets that are disjoint.
The Kneser graphs of the form~$K(2k+1,k)$ are also known as the \emph{odd graphs}.
We settle an old problem due to Meredith, Lloyd, and Biggs from the 1970s, proving that for every~$k\geq 3$, the odd graph~$K(2k+1,k)$ has a Hamilton cycle.
This and a known conditional result due to Johnson imply that all Kneser graphs of the form~$K(2k+2^a,k)$ with~$k\geq 3$ and~$a\geq 0$ have a Hamilton cycle.
We also prove that~$K(2k+1,k)$ has at least~$2^{2^{k-6}}$ distinct Hamilton cycles for~$k\geq 6$.
Our proofs are based on a reduction of the Hamiltonicity problem in the odd graph to the problem of finding a spanning tree in a suitably defined hypergraph on Dyck words.
\end{abstract}

\maketitle

\section{Introduction}

The question whether a given graph has a Hamilton cycle is one of the oldest and most fundamental problems in graph theory and computer science, shown to be NP-complete in Karp's seminal paper~\cite{MR0378476}.
The problem originates from the 19th-century ``Hamilton puzzle'', which involves finding a Hamilton cycle along the edges of a dodecahedron.
Efficient methods of generating Hamilton cycles in highly symmetric graphs (in particular, so-called Gray codes) are particularly important from the point of view of practical applications~\cite{MR3444818,MR1491049}.
Still, for various natural and extensively studied families of graphs, it is conjectured that a Hamilton cycle always exists, but finding one is a notoriously hard problem; see for instance~\cite{MR4075363,DBLP:journals/talg/SawadaW20}.
In this paper, we focus on a well-known instance of this phenomenon---the so-called Kneser graphs.

\subsection{Kneser graphs}
\label{sec:Knk}

For any two integers~$k\geq 1$ and~$n\geq 2k+1$, the \emph{Kneser graph}~$K(n,k)$ has the $k$-element subsets of $[n]:=\{1,\ldots,n\}$ as vertices and the pairs of those subsets that are disjoint as edges.
These graphs were introduced by Lov\'asz in his celebrated proof of Kneser's conjecture~\cite{MR514625}.
The proof uses topological methods to show that the chromatic number of~$K(n,k)$ is equal to~$n-2k+2$.
Lov\'asz's result initiated an exciting line of research~\cite{MR514626,MR1941810,MR2057690,MR1893009} and gave rise to the nowadays flourishing fields of topological combinatorics and computational topology, see e.g.~\cite{MR3210821,MR3215297}.
Apart from the above, Kneser graphs have many other interesting properties.
For instance, the maximum size of an independent set in~$K(n,k)$ is equal to~$\binom{n-1}{k-1}$, by the famous Erd\H{o}s-Ko-Rado theorem~\cite{MR0140419}.

\subsection{Hamilton cycles in Kneser graphs}

As indicated before, it has long been conjectured that Kneser graphs have Hamilton cycles.
Apart from one obvious exception, namely the Petersen graph~$K(5,2)$ shown in Figure~\ref{fig:K52}, no other negative instances are apparent.
Observe that Kneser graphs are vertex-transitive, that is, they look the same from the point of view of any vertex.
This makes them an excellent test case for a famous and vastly more general conjecture due to Lov\'asz~\cite{MR0263646}, which asserts that any connected and vertex-transitive graph has a Hamilton cycle, apart from the Petersen graph and four other exceptional instances.

We proceed by giving an account of the long history of finding Hamilton cycles in Kneser graphs.
The degree of every vertex in~$K(n,k)$ is~$\binom{n-k}{k}$, so for fixed~$k$, increasing~$n$ also increases the vertex degrees, which intuitively makes the task of finding a Hamilton cycle easier.
The density is also witnessed by cliques of size~$c\geq 3$, which are present for~$n\geq ck$ and absent for~$n<ck$.
The sparsest case, for which finding a Hamilton cycle is intuitively hardest, is when~$n=2k+1$.
The corresponding graphs $O_k:=K(2k+1,k)$, for~$k\geq 1$, are known as \emph{odd graphs}.
They include the Petersen graph~$O_2=K(5,2)$.
The odd graphs~$O_2$ and~$O_3$ are illustrated in Figures~\ref{fig:K52} and~\ref{fig:K73}, respectively.
Note that all vertices in the odd graph~$O_k$ have degree~$k+1$, which is only logarithmic in the number of vertices.
The conjecture that the odd graph~$O_k$ has a Hamilton cycle for every~$k\geq 3$ originated in the 1970s, in papers by Meredith and Lloyd~\cite{MR0457282,MR0321782} and by Biggs~\cite{MR556008}.
A~stronger version of the conjecture asserts that~$O_k$ even has~$\lfloor (k+1)/2\rfloor$ edge-disjoint Hamilton cycles.
Already Balaban~\cite{balaban:72} exhibited a Hamilton cycle for the cases~$k=3$ and~$k=4$, and Meredith and Lloyd described one for~$k=5$ and~$k=6$.
Later, Mather~\cite{MR0389663} also solved the case~$k=7$.
With the help of computers, Shields and Savage~\cite{MR2020936} found Hamilton cycles in~$O_k$ for all values of~$k$ up to~$13$.
They also found Hamilton cycles in~$K(n,k)$ for all~$n\leq 27$ (except for the Petersen graph).

\begin{figure}[t]
\includegraphics[scale=0.898]{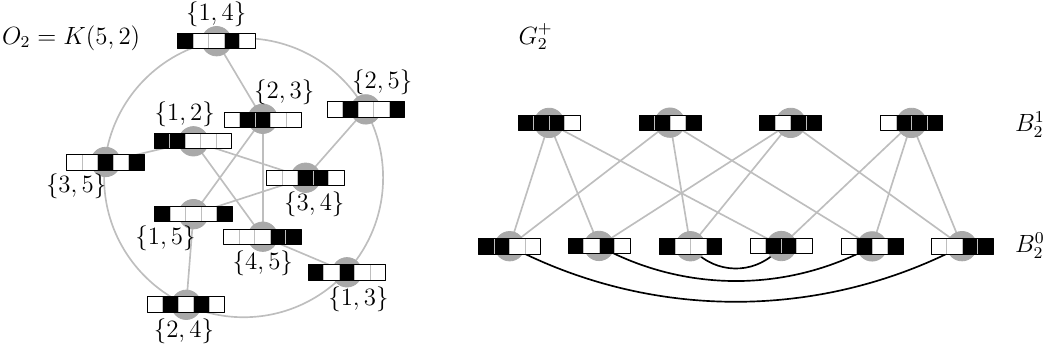}
\caption{The Petersen graph~$O_2=K(5,2)$ (left) and the graph~$G_2^+$ (right) that is isomorphic to it.
The isomorphism is defined in the proof of Lemma~\ref{lem:Ok-cube}.
The vertices of the Petersen graph are shown as $2$-element subsets of $[5]=\{1,2,3,4,5\}$, together with the corresponding characteristic bitstrings.
Black squares represent $1$-bits and white squares represent $0$-bits.}
\label{fig:K52}
\end{figure}

There is a long line of research devoted to proving that sufficiently dense Kneser graphs have a Hamilton cycle.
Heinrich and Wallis~\cite{MR510592} showed that~$K(n,k)$ has a Hamilton cycle if $n\geq 2k+k/(\sqrt[k]{2}-1)=(1+o(1))k^2/\ln 2$.
This was improved by B.~Chen and Lih~\cite{MR888679}, whose results imply that~$K(n,k)$ has a Hamilton cycle if $n\geq (1+o(1))k^2/\log k$, see~\cite{MR1405991}.
In another breakthrough, Y.~Chen~\cite{MR1778200} showed that~$K(n,k)$ is Hamiltonian when~$n\geq 3k$.
A~particularly nice and clean proof for the cases where~$n=ck$, $c\in\{3,4,\ldots\}$, was obtained by Y.~Chen and F\"uredi~\cite{MR1883565}.
Their proof uses Baranyai's well-known partition theorem for complete hypergraphs~\cite{MR0416986} to partition the vertices of~$K(ck,k)$ into cliques of size~$c$.
This proof method was extended by Bellmann and Sch\"ulke to any $n\geq 4k$~\cite{bellmann-schulke:19}.
The asymptotically best result currently known, again due to Y.~Chen~\cite{MR1999733}, is that~$K(n,k)$ has a Hamilton cycle if $n\geq (3k+1+\sqrt{5k^2-2k+1})/2=(1+o(1))2.618\ldots\cdot k$.

Another line of attack towards proving Hamiltonicity is to find long cycles in~$K(n,k)$.
To this end, Johnson~\cite{MR2046083} showed that there exists a constant~$c>0$ such that the odd graph~$O_k$ has a cycle that visits at least a $1-c/\sqrt{k}$~proportion of all vertices, which is almost all vertices as $k$~tends to infinity.
This was generalized and improved in~\cite{MR3759914}, where it was shown that~$K(n,k)$ has a cycle visiting a $2k/n$~proportion of all vertices.
The last result implies that~$O_k$ has a cycle visiting a $1-1/(2k+1)$~proportion of the vertices (e.g., the Petersen graph~$O_2$ has a cycle that visits~$8$ of its~$10$ vertices).

A~different relaxation of proving Hamiltonicity is to construct a cycle factor, that is, a collection of vertex-disjoint cycles that together cover all vertices of the graph.
From this point of view, a Hamilton cycle is a cycle factor consisting of a single cycle.
In this direction, Johnson and Kierstead~\cite{MR2128031} showed that the edges of~$O_k$ can be partitioned into cycle factors for odd~$k$ and into cycle factors and one matching for even~$k$.
A~different cycle factor in~$O_k$, which turns out to be crucial for our present result, was constructed in~\cite{MR3738156}.
It is shown in Figure~\ref{fig:K73} for the case~$k=3$.

\subsection{Bipartite Kneser graphs}
\label{sec:Hnk}

Bipartite Kneser graphs form another family of vertex-transi\-tive graphs closely related to Kneser graphs.
The \emph{bipartite Kneser graph}~$H(n,k)$ has all $k$-element and all $(n-k)$-element subsets of~$[n]$ as vertices and all pairs of these subsets such that one is contained in the other as edges.
It has been a long-standing problem to show that~$H(n,k)$ has a Hamilton cycle.
A~detailed account of the historic developments is given in~\cite{MR3759914}.
Also here, the sparsest case~$H(2k+1,k)$ resisted all attacks for more than three decades, and the question whether~$H(2k+1,k)$ has a Hamilton cycle became known as the \emph{middle levels conjecture}.
This conjecture has been recently solved affirmatively in~\cite{MR3483129} (see also \cite{gregor-muetze-nummenpalo:18}), and the general case, the Hamiltonicity of~$H(n,k)$, has been settled subsequently in~\cite{MR3759914}.
Note that proving Hamiltonicity for the Kneser graph~$K(n,k)$ is arguably harder than for the bipartite Kneser graph~$H(n,k)$.
In particular, proving that the odd graphs $O_k=K(2k+1,k)$ are Hamiltonian is harder than the middle levels conjecture.
Specifically, from a Hamilton cycle $(x_1,\ldots,x_N)$ in~$K(n,k)$, where $N=\binom{n}{k}$, we can easily construct a Hamilton cycle or a Hamilton path in~$H(n,k)$, as follows.
Consider the sequences $C_1:=(x_1,\ol{x_2},x_3,\ol{x_4},\ldots)$ and $C_2:=(\ol{x_1},x_2,\ol{x_3},x_4,\ldots)$, where $\ol{x_i}:=[n]\setminus x_i$.
If $N$~is odd, then~$C_1$ and~$C_2$ together form a Hamilton cycle in~$H(n,k)$.
If $N$~is even, then~$C_1$ and~$C_2$ are two cycles in~$H(n,k)$ that can be joined to form a Hamilton path.
In fact, the arguments given in this paper easily give a Hamilton cycle in~$H(2k+1,k)$ for all~$k\geq 1$, providing an alternative proof of the middle levels conjecture; see Section~\ref{sec:mlc}.

\subsection{Our results}

We prove that the odd graphs $O_k=K(2k+1,k)$ with $k\geq 3$ contain Hamilton cycles.
That is, we resolve the sparsest case of the conjecture on the Hamiltonicity of Kneser graphs in the affirmative.

\begin{theorem}
\label{thm:odd}
For any integer\/~$k\geq 3$, the odd graph\/ $O_k=K(2k+1,k)$ has a Hamilton cycle.
\end{theorem}

Using the conditional results proved by Johnson~\cite{MR2836824}, Theorem~\ref{thm:odd} immediately yields the following more general statement.

\begin{theorem}
\label{thm:sparse}
For any integers\/~$k\geq 3$ and\/~$a\geq 0$, the Kneser graph\/~$K(2k+2^a,k)$ has a Hamilton cycle.
\end{theorem}

We also establish the following counting version of Theorem~\ref{thm:odd}.

\begin{theorem}
\label{thm:count}
For any integer\/~$k\geq 6$, the odd graph\/ $O_k=K(2k+1,k)$ has at least\/~$2^{2^{k-6}}$ distinct Hamilton cycles.
\end{theorem}

The double-exponential growth of the number of Hamilton cycles guaranteed by Theorem~\ref{thm:count} is essentially best possible: since~$O_k$ has~$\smash[b]{\binom{2k+1}{k}}$ vertices, the number of Hamilton cycles in~$O_k$ is at most $\binom{2k+1}{k}!=2^{2^{\cO(k)}}$.
Note also that applying automorphisms of~$O_k$ to a a single Hamilton cycle yields at most $(2k+1)!=2^{\Theta(k\log k)}$ distinct Hamilton cycles, substantially fewer than guaranteed by Theorem~\ref{thm:count}.
In other words, Theorem~\ref{thm:count} is not an immediate consequence of Theorem~\ref{thm:odd}.

\subsection{Gray code algorithms}
\label{sec:algo}

Hamilton cycles in Kneser graphs and bipartite Kneser graphs are closely related to Gray codes.
A~\emph{combinatorial Gray code} is the algorithmic problem of generating all objects in a combinatorial class, such as bitstrings, permutations, combinations, partitions, trees, or triangulations, etc., in some well-defined order.
Gray codes have found widespread use in areas such as circuit testing, signal encoding, data compression, graphics, and image processing etc.---see the survey~\cite{MR1491049} and the references therein.
The ultimate goal for Gray code algorithms is to generate each new object from the previous one in constant time, which entails that consecutive objects may differ only by a constant amount.
A~Gray code thus corresponds to a Hamilton cycle in a graph whose vertices are the combinatorial objects and whose edges connect objects that differ only by such an elementary transformation.
More than half of the most recent volume of Knuth's seminal series \emph{The Art of Computer Programming}~\cite{MR3444818} is devoted to this fundamental subject.
The two hardest Gray code problems mentioned in Knuth's book (Problem~71 in Section~7.2.1.2 and Problem~56 in Section~7.2.1.3), including the middle levels conjecture, have been solved in the meantime, and efficient algorithms to generate these Gray code have been developed in~\cite{MR4075363} and~\cite{DBLP:journals/talg/SawadaW20}.
Recall from Section~\ref{sec:Hnk} that Hamiltonicity of the odd graphs is arguably harder than the middle levels conjecture.

Our proof of Theorem~\ref{thm:odd} is constructive and translates straightforwardly into an algorithm to compute a Hamilton cycle in the odd graph~$O_k$ in polynomial time (polynomial in the size of the graph, which is exponential in~$k$).
We can identify each $k$-element subset of~$[2k+1]$ with a bitstring of length~$2k+1$, where the $i$th bit is set to~$1$ if the element~$i$ is contained in the set and it is set to~$0$ otherwise; see Figure~\ref{fig:K52}.
A~Hamilton cycle in the odd graph thus corresponds to a Gray code listing of all bitstrings of length~$2k+1$ with exactly $k$~many $1$-bits, such that consecutive bitstrings differ in all but one position.
It remains open whether our proof can be translated into a constant-time algorithm to generate this Gray code, that is, an algorithm that in each step computes the bit that is not flipped in constant time, using only $\cO(k)$~memory space and polynomial initialization time.
To avoid costly complementation operations, such an algorithm could maintain two bitstrings, one the complement of the other, along with a flag indicating which of the two bitstrings is the current one; then, in each step, only a single bit in both bitstrings and the flag would need to be flipped.

\subsection{Proof idea}
\label{sec:idea}

\begin{figure}[t]
\includegraphics[scale=1]{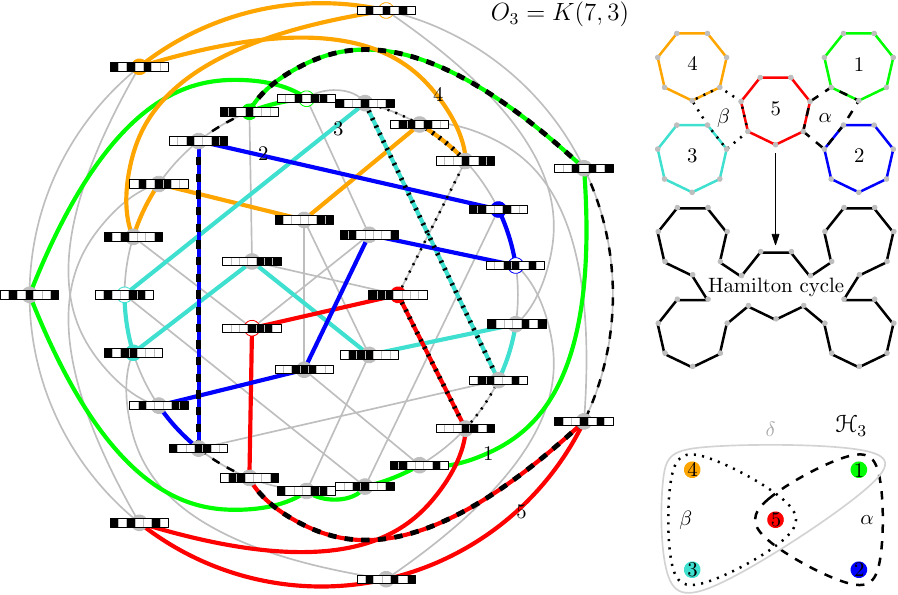}
\caption{Illustration of our Hamiltonicity proof for the odd graph~$O_3=K(7,3)$.
The vertices are represented as bitstrings, where $1$-bits are drawn as black squares and $0$-bits as white squares.
The bold cycles labeled~$1,\ldots,5$ constitute the cycle factor~$\cC_3$.
Two flipping cycles of length~$6$ are highlighted by dashed and dotted lines.
The five cycles from~$\cC_3$ correspond to the vertices of the hypergraph~$\cH_3$, and the two flipping cycles correspond to the hyperedges~$\alpha$ and~$\beta$.
There is another flipping cycle of length~$8$ in the graph~$O_3$, corresponding to the hyperedge~$\delta$ in~$\cH_3$, but this cycle is not shown in the figure.
As~$\{\alpha,\beta\}$ is a spanning tree in~$\cH_3$, taking the symmetric difference of the edge sets of the cycles in~$\cC_3$ with the edge sets of the two corresponding flipping cycles indicated in the figure yields a Hamilton cycle in the graph~$O_3$.
This construction step is shown schematically in the top right part of the figure.}
\label{fig:K73}
\end{figure}

We construct a Hamilton cycle in the odd graph~$O_k$ as follows; see Figure~\ref{fig:K73}.
We start with the cycle factor~$\cC_k$ in the odd graph~$O_k$ described in~\cite{MR3738156}.
It has the property that all of its cycles have the same length~$2k+1$ and the number of cycles is the $k$th Catalan number.
Furthermore, the cycles in~$\cC_k$ can be identified with so-called \emph{Dyck words} of length~$2k$, that is, bitstrings of length~$2k$ with the property that every prefix has at least as many $1$-bits as $0$-bits.
It is well known that the number of such Dyck words is equal to the $k$th Catalan number~\cite{MR1676282}.

Given the cycle factor~$\cC_k$, we modify it locally to join its cycles into a single Hamilton cycle in~$O_k$.
Each such modification involves $\ell$~cycles $C_1,\ldots,C_\ell$ from the factor~$\cC_k$ and a $2\ell$-cycle~$C'$ that shares exactly one edge with each of $C_1,\ldots,C_\ell$.
Specifically, $C'$~shares every second of its edges with one of the $\ell$~cycles, and every other edge of~$C'$ goes between two different cycles.
Consequently, taking the symmetric difference of the edge set of~$C'$ with the edge sets of $C_1,\ldots,C_\ell$ yields a single cycle on the vertex set of all $C_1,\ldots,C_\ell$.
We call a cycle~$C'$ with this property a \emph{flipping cycle}.
In Figure~\ref{fig:K73}, two flipping $6$-cycles are highlighted with dashed and dotted lines.
We perform this operation simultaneously with an appropriate set of mutually edge-disjoint flipping cycles so as to join all cycles in~$\cC_k$ into a single cycle.
Although the joining operation can work with flipping $2\ell$-cycles for any~$\ell\geq 2$, we will use only $6$-cycles ($\ell=3$) and $8$-cycles ($\ell=4$).
We cannot use flipping $4$-cycles ($\ell=2$), because the odd graph~$O_k$ has no $4$-cycles at all.

This approach can be formalized as follows.
We construct a hypergraph~$\cH_k$ whose vertices are the Dyck words of length~$2k$ representing the cycles of the factor~$\cC_k$.
Each $\ell$-edge ($3$-edge or $4$-edge) of~$\cH_k$ represents a flipping $2\ell$-cycle ($6$-cycle or $8$-cycle, respectively) that can be used to join $\ell$~cycles from~$\cC_k$ as described before.
In the example illustrated in Figure~\ref{fig:K73}, the hypergraph~$\cH_3$ consists of three hyperedges labeled~$\alpha$, $\beta$, and~$\delta$ of cardinalities~$3$, $3$, and~$4$, respectively.
Here is the key insight about the hypergraph~$\cH_k$: in order to prove that the odd graph~$O_k$ has a Hamilton cycle, it suffices to prove that the hypergraph~$\cH_k$ has a \emph{spanning tree}, that is, a connected and acyclic set of hyperedges covering all vertices.
In such a spanning tree, any two hyperedges intersect in at most one element.
For instance, the hypergraph~$\cH_3$ in Figure~\ref{fig:K73} has a spanning tree~$\{\alpha,\beta\}$.
The hypergraph~$\cH_k$ that we construct has the property that the flipping cycles represented by the hyperedges in any spanning tree are mutually edge-disjoint.
Consequently, every spanning tree in~$\cH_k$ corresponds to a collection of flipping cycles such that the symmetric difference of their edge sets and the edges of the cycles in~$\cC_k$ results in a Hamilton cycle in the odd graph~$O_k$.

The proof of Theorem~\ref{thm:count} exploits the degrees of freedom that are inherent in the construction above to provide double-exponentially many distinct spanning trees in~$\cH_k$, which give rise to double-exponentially many distinct Hamilton cycles in~$O_k$.
This general approach of reducing a Hamilton cycle problem to a spanning tree problem in a suitably defined auxiliary (hyper)graph has also been exploited in several other papers; see e.g.~\cite{gregor-muetze-nummenpalo:18,MR3854107,MR3599935,MR2925746,MR2548540,MR2836824,MR3483129,DBLP:journals/talg/SawadaW20}.

\subsection{Outline of the paper}

In Section~\ref{sec:preliminaries}, we introduce notation and terminology that will be used throughout this paper, and we recall the construction of the cycle factor~$\cC_k$ given in~\cite{MR3738156}.
In Section~\ref{sec:flip}, we describe how the cycles in~$\cC_k$ are joined to form a Hamilton cycle in~$O_k$, and we present the proofs of Theorems~\ref{thm:odd}--\ref{thm:count}.
The proofs of some technical lemmas are deferred to Sections~\ref{sec:flip-lemmas} and~\ref{sec:spanning}.
In Section~\ref{sec:mlc}, we give an alternative proof of the middle levels conjecture.
We conclude with some open problems in Section~\ref{sec:open}.

\section{Preliminaries}
\label{sec:preliminaries}

\subsection{Bitstrings and Dyck paths}

A~\emph{bitstring} is a finite sequence of digits~$0$ and~$1$ called the \emph{bits} of the bitstring.
The empty bitstring is denoted by~$\epsilon$.
The concatenation of two bitstrings~$x$ and~$y$ is denoted by~$xy$.
For every bitstring~$x$, we define $x^0:=\epsilon$ and $x^n:=x^{n-1}x$ for~$n\geq 1$.
The length of a bitstring~$x$ is denoted by~$|x|$.
The \emph{complement} of a bitstring~$x$, denoted by~$\ol{x}$, is the bitstring obtained from~$x$ by \emph{flipping} every bit, that is, by replacing every $1$-bit by a $0$-bit and vice versa.

The \emph{weight} of a bitstring~$x$ is the number of $1$-bits in~$x$.
We let~$B_k^0$ and~$B_k^1$ denote the sets of bitstrings of length~$2k$ with weights~$k$ and~$k+1$, respectively, and we let $B_k:=B_k^0\cup B_k^1$.
It follows that $|B_k^0|=\smash[t]{\binom{2k}{k}}$, $|B_k^1|=\smash[t]{\binom{2k}{k-1}}$, and $|B_k|=\smash[t]{\binom{2k+1}{k}}$.
We let~$D_k$ denote the set of bitstrings of length~$2k$ with weight~$k$ and with the property that in every prefix, the number of $1$-bits is at least the number of $0$-bits.
It is a well known fact that $|D_k|=\frac{1}{k+1}\binom{2k}{k}=\frac{1}{2k+1}\binom{2k+1}{k}$, which is the $k$th Catalan number.
We also define $D:=\bigcup_{k=0}^\infty D_k$, and we call every bitstring in~$D$ a \emph{Dyck word}.

It is sometimes convenient to represent a Dyck word~$x\in D_k$ by a \emph{Dyck path} of length~$2k$ in the integer lattice~$\mathbb{Z}^2$.
Every $1$-bit in the Dyck word~$x$ is represented by an \emph{up-step}, which changes the current coordinates by~$(+1,+1)$, and every $0$-bit is represented by a \emph{down-step}, which changes the current coordinates by~$(+1,-1)$; see Figure~\ref{fig:pi}.
The prefix property from the definition of~$D_k$ corresponds to the property that the lattice path never goes below the abscissa.

For a Dyck word $x=b_1b_2\cdots b_{2k}\in D_k$, where $b_1,\ldots,b_{2k}\in\{0,1\}$, we let $\revinv{x}:=\ol{b_{2k}b_{2k-1}\cdots b_1}$.
That is, $\revinv{x}$ is the complement of the reverse of~$x$, which is itself a Dyck word in~$D_k$.
For example, if $x=110010$, then $\revinv{x}=101100$.
We call the operation $x\mapsto\revinv{x}$ \emph{mirroring}.
In terms of Dyck path representation, it corresponds to taking the mirror image with respect to the vertical line~$x=k$.

\subsection{Graphs \texorpdfstring{$G_k$}{Gk} and \texorpdfstring{$G_k^+$}{Gk+}}

We use standard graph-theoretic terminology, where the edges of every graph that we consider are unordered pairs of vertices of the form~$\{u,v\}$.
We define~$G_k$ as the graph with vertex set~$B_k$ and with edges that connect pairs of bitstrings that differ by exactly one bit.
In other words, $G_k$ is the subgraph of the $2k$-dimensional hypercube induced by the bitstrings with weights~$k$ and~$k+1$.
For~$k\geq 1$, we also define~$G_k^+$ as the graph obtained from~$G_k$ by adding all edges of the form~$\{x,\ol{x}\}$ where~$x\in B_k^0$.
This construction is illustrated on the right hand side of Figure~\ref{fig:K52}, where the edges~$\{x,\ol{x}\}$ are highlighted in black.
Observe that while the graph~$G_k$ is bipartite, the graph~$G_k^+$ is not.

\begin{lemma}
\label{lem:Ok-cube}
For every\/~$k\geq 1$, the graph\/~$G_k^+$ is isomorphic to the odd graph\/~$O_k$.
\end{lemma}

\begin{proof}
A~natural isomorphism between~$G_k^+$ and~$O_k$ is obtained by mapping every~$x\in B_k^0$ to~$x0$ and every~$x\in B_k^1$ to~$\overline{x}1$ and by interpreting the resulting bitstrings of length~$2k+1$ and weight~$k$ as characteristic vectors of $k$-element subsets of~$[2k+1]$.
It is straightforward to verify that this mapping preserves edges and non-edges.
\end{proof}

To prove Theorems~\ref{thm:odd} and~\ref{thm:count}, we will use Lemma~\ref{lem:Ok-cube} and construct Hamilton cycles in~$G_k^+$ for all~$k\geq 3$.

\subsection{Cycle factor \texorpdfstring{$\cC_k$}{Ck} in \texorpdfstring{$G_k^+$}{Gk+}}
\label{sec:2factor}

A~\emph{cycle factor} in a graph is a collection of vertex-disjoint cycles that together cover all vertices of the graph.
The cycle factor~$\cC_k$ in~$G_k^+$, which we will define shortly, was introduced and analyzed in~\cite{MR3738156}.
The cycles in~$\cC_k$ correspond to Dyck words in~$D_k$ as follows.
For every Dyck word~$x\in D_k$, we define a permutation~$\pi(x)$ of the set~$[2k]$.
Then, we define a path~$P(x)$ in~$G_k$ whose subsequent vertices are obtained by starting from~$x$ and flipping the bits one by one at positions determined by the sequence $\pi(x)=(a_1,\ldots,a_{2k})$, ending at~$\ol{x}$.
Finally, we add the edge~$\{x,\ol{x}\}$ to~$P(x)$, obtaining a cycle~$C(x)$ in~$G_k^+$ that becomes a member of~$\cC_k$.

We let $(a_1,\ldots,a_n)$ denote the sequence of integers $a_1,\ldots,a_n$.
We generalize this notation allowing~$a_i$ to be itself an integer sequence---in that case, if $a_i=(b_1,\ldots,b_m)$, then $(a_1,\ldots,a_n)$ is shorthand for $(a_1,\ldots,a_{i-1},b_1,\ldots,b_m,a_{i+1},\ldots,a_n)$.
The empty integer sequence is denoted by~$()$.
For an integer sequence $\pi=(a_1,\ldots,a_n)$ and an integer~$a$, we define
\begin{equation*}
a+\pi:=(a+a_1,\ldots,a+a_n), \hspace{4em} a-\pi:=(a-a_1,\ldots,a-a_n).
\end{equation*}

\begin{figure}[t]
\includegraphics[scale=0.916]{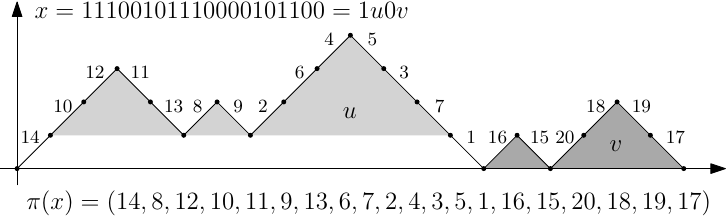}
\caption{Dyck path representation of a Dyck word~$x\in D_{10}$ and the permutation~$\pi(x)$.
The numbers on top of the Dyck path edges represent the order in which its up-steps and down-steps occur in the sequence~$\pi(x)$, i.e., they are equal to the inverse permutation~$(\pi(x))^{-1}$ when read from left to right.}
\label{fig:pi}
\end{figure}

It is clear that every non-empty Dyck word~$x\in D$ has a unique decomposition of the form~$x=1u0v$, where~$u,v\in D$; see Figure~\ref{fig:pi}.
Using this fact, for every Dyck word~$x\in D$, we define an integer sequence~$\pi(x)$ of length~$|x|$ as follows, by induction on~$|x|$:
\begin{equation}
\label{eq:pi}
\begin{aligned}
\pi(\epsilon) &:= (),\\
\pi(1u0v)     &:= \bigl(|u|+2,\:(|u|+2)-\pi(\revinv{u}),\:1,\:(|u|+2)+\pi(v)\bigr) \quad \text{for any }u,v\in D.
\end{aligned}
\end{equation}
The sequence~$\pi(x)$ for~$x\in D_k$ satisfies the following properties:
\begin{enumerate}
\item\label{prop:paths-i} $\pi(x)$ is a permutation of the set~$[2k]$;
\item\label{prop:paths-ii} if $\pi(x)=(a_1,\ldots,a_{2k})$, then the bit of~$x$ at position~$a_i$ is~$0$ for $i$~odd and~$1$ for $i$~even.
\end{enumerate}
To see why, we apply induction on~$|x|$.
The base case $x=\epsilon$ clearly satisfies both~\ref{prop:paths-i} and~\ref{prop:paths-ii}.
If $x\neq\epsilon$, then the sequence $\pi(x) = \pi(1u0v)$ is, by definition, a concatenation of four sequences that are, by induction, permutations of the sets $\{|u|+2\}$, $\{(|u|+2)-|u|,\ldots,(|u|+2)-1\}$, $\{1\}$, and $\{(|u|+2)+1,\ldots,(|u|+2)+|v|\}$.
These sets form a partition of~$[2k]$, which proves~\ref{prop:paths-i}.
To prove~\ref{prop:paths-ii}, we distinguish on which of the four aforementioned sets $a_i$~belongs to.
Suppose $a_i\in\{2,\ldots,|u|+1\}$.
It follows from \ref{prop:paths-i} and~\eqref{eq:pi} that $a_i$~is the $(i-1)$th entry of the sequence $(|u|+2)-\pi(\revinv{u})$.
Let $\pi(\revinv{u}) = (b_1,\ldots,b_{|u|})$.
By the induction hypothesis, the bit of~$\revinv{u}$ at position~$b_{i-1}$ is~$0$ if and only if $i$~is even.
The bit of~$x$ at position~$a_i$ is the bit of~$u$ at position $a_i-1=(|u|+1)-b_{i-1}$, which is the complement of the bit of~$\revinv{u}$ at position~$b_{i-1}$.
Therefore, the bit of~$x$ at position~$a_i$ is~$1$ if and only if $i$~is even, as claimed in~\ref{prop:paths-ii}.
We leave the analysis of the remaining cases to the reader.

In terms of Dyck path representation, we can interpret~$\pi(x)$ as the alternating order of down-steps and up-steps of the Dyck path~$x$; see Figure~\ref{fig:pi}.
The first term of~$\pi(x)$ represents the first down-step that touches the abscissa---it goes from~$(|u|+1,1)$ to~$(|u|+2,0)$.
The next part of~$\pi(x)$ represents the up-steps and down-steps of the part~$u$ of the Dyck path between~$(1,1)$ to~$(|u|+1,1)$ in the order obtained recursively on the mirror image of~$u$.
The next term of~$\pi(x)$ represents the first up-step, which goes from~$(0,0)$ to~$(1,1)$.
The final part of~$\pi(x)$ represents the down-steps and up-steps of the part~$v$ of the Dyck path between~$(|u|+2,0)$ to~$(|x|,0)$ ordered recursively.

Now, let~$x\in D_k$ and $\pi(x)=(a_1,\ldots,a_{2k})$.
Using the properties~\ref{prop:paths-i} and~\ref{prop:paths-ii} above, we define a path $P(x)=(x_0,x_1,\ldots,x_{2k})$ in the graph~$G_k$ so that~$x_0=x$ and~$x_i$ is obtained from~$x_{i-1}$ by flipping the bit at position~$a_i$ for every~$i\in[2k]$, whence it follows that~$x_{2k}=\ol{x}$.
We call~$\pi(x)$ the \emph{bit-flip sequence} for~$P(x)$.
We define the set of paths~$\cP_k$ by
\begin{equation*}
\cP_k := \{P(x)\mid x\in D_k\}.
\end{equation*}
The set of paths~$\cP_3$ with the corresponding bit-flip sequences is illustrated in Figure~\ref{fig:P3}.

The following lemma is a consequence of the results of~\cite{MR3738156}.
For the reader's convenience, we provide a short direct proof of it in Section~\ref{sec:flip-lemmas}.

\begin{lemma}[\cite{MR3738156}]
\label{lem:paths}
For every\/~$k\geq 1$, the paths in\/~$\cP_k$ are mutually vertex-disjoint, and together they cover all vertices of\/~$G_k$.
\end{lemma}

For every Dyck word~$x\in D_k$ and every bit position~$i\in[2k]$, we let~$e(x,i)$ denote the edge of the path~$P(x)$ along which the $i$th bit is flipped.
That is, if $\pi(x)=(a_1,\ldots,a_{2k})$, then the path~$P(x)$ contains edges $e(x,a_1),\ldots,e(x,a_{2k})$ in this order along the path from~$x$ to~$\ol{x}$.
For example, for~$x_1$ as in Figure~\ref{fig:P3}, we have $e(x_1,3)=\{101101,100101\}$ and $e(x_1,1)=\{100111,000111\}$.

For every Dyck word~$x\in D_k$ with~$k\geq 1$, the first vertex~$x$ and the last vertex~$\ol{x}$ of~$P(x)$ are adjacent in~$G_k^+$.
We let~$C(x)$ denote the cycle in~$G_k^+$ obtained by adding the edge~$\{x,\ol{x}\}$ to the path~$P(x)$.
We define
\begin{equation*}
\cC_k := \{C(x)\mid x \in D_k\}.
\end{equation*}
It follows from Lemma~\ref{lem:paths} that the set of cycles~$\cC_k$ is a cycle factor in~$G_k^+$.
Figure~\ref{fig:K73} illustrates the cycles in~$\cC_3$, which are obtained by closing the paths in~$\cP_3$ illustrated in Figure~\ref{fig:P3} and applying the isomorphism between~$G_k^+$ and~$O_k$ described in the proof of Lemma~\ref{lem:Ok-cube}.

\begin{figure}[t]
\newcounter{nodecount}
\newcommand\tabnode[1]{\addtocounter{nodecount}{1}\tikz \node (\arabic{nodecount}) {#1};}
\tikzstyle{every picture}+=[remember picture,baseline]
\tikzstyle{every node}+=[inner sep=0.5pt,anchor=base,minimum width=1.24cm,align=center,outer sep=1.5pt]
\begin{center}
\setlength{\tabcolsep}{7pt}
\begin{tabular}{cC|cC|cC|cC|cC}
\multicolumn{2}{c|}{\cellcolor{red!50}$x_1=111000$} & \multicolumn{2}{c|}{\cellcolor{green!50}$x_2=110100$} & \multicolumn{2}{c|}{\cellcolor{blue!50}$x_3=110010$} & \multicolumn{2}{c|}{\cellcolor{cyan!30}$x_4=101100$} & \multicolumn{2}{c}{\cellcolor{orange!50}$x_5=101010$} \\
$P(x_1)$ & \multicolumn{1}{c|}{$\pi(x_1)$} & $P(x_2)$ & \multicolumn{1}{c|}{$\pi(x_2)$} & $P(x_3)$ & \multicolumn{1}{c|}{$\pi(x_3)$} & $P(x_4)$ & \multicolumn{1}{c|}{$\pi(x_4)$} & $P(x_5)$ & \multicolumn{1}{c}{$\pi(x_5)$} \\ \hline
&&&&&&&\\[-2ex] 
\tabnode{111000} & 6 & \tabnode{110100} & 6 & 110010           & 4 & 101100           & 2 & 101010           & 2 \\
\tabnode{111001} & 2 & \tabnode{110101} & 4 & \tabnode{110110} & 2 & 111100           & 1 & \tabnode{111010} & 1 \\
\tabnode{101001} & 4 & \tabnode{110001} & 5 & \tabnode{100110} & 3 & 011100           & 6 & \tabnode{011010} & 4 \\
101101           & 3 & \tabnode{110011} & 2 & 101110           & 1 & 011101           & 4 & 011110           & 3 \\
\tabnode{100101} & 5 & 100011           & 3 & 001110           & 6 & \tabnode{011001} & 5 & 010110           & 6 \\
\tabnode{100111} & 1 & 101011           & 1 & 001111           & 5 & \tabnode{011011} & 3 & 010111           & 5 \\
000111           &   & 001011           &   & 001101           &   & \tabnode{010011} &   & 010101           &   \\
\end{tabular}
\begin{tikzpicture}[overlay]
\draw [very thick,dashed,rounded corners=3.5pt] (12.north west) rectangle (14.south east);
\draw [very thick,dashed,rounded corners=3.5pt] (2.north west) rectangle (4.south east);
\draw [very thick,dashed,rounded corners=3.5pt] (5.north west) rectangle (9.south east);
\draw [thick,rounded corners=3.5pt] (1.north west) rectangle (3.south east);
\draw [thick,rounded corners=3.5pt] (8.north west) rectangle (11.south east);
\draw [thick,rounded corners=3.5pt] (15.north west) rectangle (16.south east);
\draw [thick,rounded corners=3.5pt] (6.north west) rectangle (10.south east);
\draw [very thick,dotted,rounded corners=5.3pt] ([xshift=-2.1pt,yshift=2.1pt] 1.north west) rectangle ([xshift=2.1pt,yshift=-2.1pt] 3.south east);
\draw [very thick,dotted,rounded corners=5.3pt] ([xshift=-2.1pt,yshift=2.1pt] 13.north west) rectangle ([xshift=2.1pt,yshift=-2.1pt] 15.south east);
\draw [very thick,dotted,rounded corners=5.3pt] ([xshift=-2.1pt,yshift=2.1pt] 6.north west) rectangle ([xshift=2.1pt,yshift=-2.1pt] 10.south east);
\end{tikzpicture}
\end{center}
\caption{The set of paths $\cP_3=\{P(x_1),\ldots,P(x_5)\}$ in the graph~$G_3$ together with the bit-flip sequences $\pi(x_1),\ldots,\pi(x_5)$ that generate them.
The edges on the three flipping cycles that witness the flippable tuples~$\alpha(\epsilon)$, $\beta$, and~$\delta$ defined in~\eqref{eq:patterns} are indicated by dashed, dotted, and solid frames, respectively.}
\label{fig:P3}
\end{figure}

\section{Construction of a Hamilton cycle}
\label{sec:flip}

We describe how to modify the cycle factor~$\cC_k$ to join its cycles to a single Hamilton cycle.
As indicated in Section~\ref{sec:idea}, the modification operation consists in taking the symmetric difference with a carefully chosen set of cycles of length~$6$ or~$8$.
The key ingredient of our argument is Lemma~\ref{lem:reduction} below, which reduces the Hamiltonicity problem to a spanning tree problem in a suitably defined hypergraph.
To make these ideas formal, we introduce a few definitions.

A~\emph{flipping cycle} on~$D_k$ is a cycle in~$G_k$ of length~$2\ell$ that has exactly $\ell$~edges in common with $\ell$~distinct paths in the set~$\cP_k$ (one common edge with each path), and along the cycle these $\ell$~common edges alternate with $\ell$~edges that go between pairs of distinct paths.
Recall that each path~$P(x)\in\cP_k$ in~$G_k$ extends to the cycle~$C(x)$ in the graph~$G_k^+$ by adding the edge~$\{x,\ol{x}\}$, so a flipping cycle has exactly $\ell$~edges in common with $\ell$~distinct cycles in the set~$\cC_k$.

A~\emph{marked Dyck word} is a non-empty Dyck word in which exactly one bit has been \emph{marked}.
More formally, a \emph{marked Dyck word} is a pair~$(x,m)$ with~$x\in D_k$ and~$m\in[2k]$ for some~$k\geq 1$, where~$m$ is the position of the marked bit in~$x$.
We simplify notation of marked Dyck words by underlining the marked bit.
For instance, $1011\underline{0}0$ denotes the marked Dyck word~$(101100,5)$.
We define prepending to, appending to, and mirroring a marked Dyck word~$(x,m)$ in a natural way, as follows:
\begin{equation}
\label{eq:append-mirror-x}
\begin{alignedat}{2}
u\,(x,m)\,v    &:= (uxv,|u|+m) \quad \text{for any bitstrings }u\text{ and }v\text{ such that }uv\in D,\\
\revinv{(x,m)} &:= (\revinv{x},|x|+1-m).
\end{alignedat}
\end{equation}
For instance, if $(x,m)=1011\underline{0}0$, then $1(x,m)010=11011\underline{0}0010$ and $\smash[t]{\revinv{(x,m)}}=1\underline{1}0010$.
In terms of Dyck path representation, $(x,m)$~is a Dyck path where the $m$th step is marked; see Figure~\ref{fig:patterns}.
Under the operations of prepending, appending, and mirroring, the marked step remains at the same relative position.

A~\emph{marked\/ $\ell$-tuple} on a set of Dyck words $X\subseteq D_k$ is an unordered $\ell$-tuple of marked Dyck words of the form $\tau=\{(x_1,m_1),\ldots,(x_\ell,m_\ell)\}$, where $x_1,\ldots,x_\ell$ are distinct Dyck words in $X$, $m_1,\ldots,m_\ell\in[2k]$, and~$\ell\geq 3$.
The set $\{x_1,\ldots,x_\ell\}\subseteq X$ is called the \emph{support} of such a marked $\ell$-tuple~$\tau$ and it is denoted by~$\supp\tau$.
The index~$m_i$ is called the \emph{mark} of~$x_i$ in~$\tau$.
A~marked $\ell$-tuple $\tau=\{(x_1,m_1),\ldots,(x_\ell,m_\ell)\}$ on~$X$ is called a \emph{flippable\/ $\ell$-tuple} on~$X$ if there is a flipping $2\ell$-cycle in~$G_k$ that contains exactly the edges $e(x_1,m_1),\ldots,e(x_\ell,m_\ell)$ of the paths $P(x_1),\ldots,P(x_\ell)$, respectively.
We say that such a flipping cycle \emph{witnesses} the flippable tuple~$\tau$.

To get an intuition for these definitions, consider the Dyck words~$x_1$, $x_2$, and~$x_3$ in the first three columns in Figure~\ref{fig:P3}.
Then $\tau=\{1110\underline{0}0,11010\underline{0},1\underline{1}0010\}$ is a marked triple on~$D_3$ with support $\supp\tau=\{x_1,x_2,x_3\}$.
In fact, it is a flippable triple on~$D_3$ witnessed by a flipping $6$-cycle $
W=(100101,100111,100110,110110,110100,110101)$ that contains the edges~$e(x_1,5)$ of~$P(x_1)$, $e(x_2,6)$ of~$P(x_2)$, and $e(x_3,2)$ of~$P(x_3)$, indicated in Figure~\ref{fig:P3} by dashed frames.
By taking the symmetric difference with~$W$, the cycles~$C(x_1)$,~$C(x_2)$, and~$C(x_3)$ become joined into a single cycle.
This observation motivates the definitions that follow.

Let~$X\subseteq D_k$ (where~$k\geq 2$), let~$\cX$ be a set of flippable tuples on~$X$, and let $\cH=(X,\cX)$.
We call such a structure~$\cH$ a \emph{flippability hypergraph} on~$X$, and we apply a few standard hypergraph-theoretic terms to~$\cH$ (as follows), although the reader should realize that the members of~$\cX$ convey the marks as extra information in addition to the standard hypergraph structure.
Thus, the \emph{subhypergraph} of~$\cH$ \emph{induced} by a non-empty set~$U\subseteq X$ is defined as
\begin{equation*}
\cH[U] := \bigl(U,\:\{\tau\in\cX\mid\supp\tau\subseteq U\}\bigr).
\end{equation*}
A~\emph{spanning tree} of~$\cH$ is a subset of~$\cX$ defined as follows, by induction on~$|X|$.
If~$|X|=1$, then the only spanning tree of~$\cH$ is the empty set.
If~$|X|\geq 2$, then a set~$\cT\subseteq\cX$ is a spanning tree of~$\cH$ if and only if there are a flippable $\ell$-tuple~$\tau\in\cT$, a partition of~$X$ into non-empty subsets $X_1,\ldots,X_\ell$, and spanning trees $\cT_1,\ldots,\cT_\ell$ of $\cH[X_1],\ldots,\cH[X_\ell]$ (respectively) such that $\cT=\{\tau\}\cup\cT_1\cup\cdots\cup\cT_\ell$ and $\lvert\supp\tau\cap X_i\rvert=1$ for each~$i\in[\ell]$.
For instance, a one-element set $\cT=\{\tau\}\subseteq\cX$ is a spanning tree of~$\cH$ if and only if~$\supp\tau=X$, and a two-element set $\cT=\{\tau_1,\tau_2\}\subseteq\cX$ is a spanning tree of~$\cH$ if and only if $\supp\tau_1\cup\supp\tau_2=X$ and $\lvert\supp\tau_1\cap\supp\tau_2\rvert=1$.
In general, straightforward induction shows that the supports of any two flippable tuples in a spanning tree have at most one element of~$X$ in common.
A~\emph{conflict-free set} in~$\cH$ is a subset~$\cU$ of~$\cX$ such that the following two conditions are satisfied:
\begin{enumerate}
\item\label{item:conflict-free-i} the supports of any two flippable tuples in~$\cU$ have at most one element of~$X$ in common;
\item\label{item:conflict-free-ii} for any two distinct flippable tuples $\tau_1,\tau_2\in\cU$ whose supports have a common element~$x\in X$, the mark of~$x$ in~$\tau_1$ is different from the mark of~$x$ in~$\tau_2$.
\end{enumerate}
As it is mentioned above, every spanning tree satisfies condition~\ref{item:conflict-free-i}, and thus a \emph{conflict-free spanning tree} is a spanning tree that additionally satisfies condition~\ref{item:conflict-free-ii}.

The following lemma is the cornerstone behind our proofs of Theorems~\ref{thm:odd} and~\ref{thm:count}.
It reduces the problem of finding a Hamilton cycle in the graph~$G_k^+$ (which is isomorphic to the odd graph~$O_k$) to the problem of finding a conflict-free spanning tree in a flippability hypergraph on~$D_k$.

\begin{lemma}
\label{lem:reduction}
Let\/~$\cH$ be a flippability hypergraph on\/~$D_k$, where\/~$k\geq 3$.
If\/~$\cH$ has a conflict-free spanning tree, then the graph\/~$G_k^+$ has a Hamilton cycle.
Moreover, distinct conflict-free spanning trees of\/~$\cH$ give rise to distinct Hamilton cycles in\/~$G_k^+$.
\end{lemma}

\begin{proof}
For every flippable tuple~$\tau$ on~$D_k$, fix a flipping cycle~$W(\tau)$ in~$G_k^+$ that witnesses~$\tau$.
For a non-empty set~$X\subseteq D_k$, let~$G_k^+[X]$ denote the subgraph of~$G_k^+$ induced by the set of all vertices of the cycles~$C(x)$ with~$x\in X$.
For a non-empty set~$X\subseteq D_k$ and a conflict-free set~$\cU$ in~$\cH[X]$, let~$S(X,\cU)$ denote the symmetric difference of the edge sets of the cycles~$C(x)$ with~$x\in X$ and the cycles~$W(\tau)$ with~$\tau\in\cU$.
We prove the following statement, which immediately yields the lemma:

\begin{claim}
Let~$X$ be a non-empty subset of~$D_k$.
For every conflict-free spanning tree~$\cT$ in~$\cH[X]$, the set~$S(X,\cT)$ forms a Hamilton cycle in~$G_k^+[X]$.
Moreover, if $\cT$ is a conflict-free spanning tree in~$\cH[X]$ and $\cU$ is a conflict-free set in~$\cH[X]$ such that $S(X,\cT)=S(X,\cU)$, then $\cT=\cU$.
\end{claim}

The proof of the claim goes by induction on~$|X|$.
If~$|X|=1$, then the empty set is the unique conflict-free spanning tree in~$\cH[X]$, and $S(X,\emptyset)=C(x)$ for the unique~$x\in X$.
For the rest of the proof, suppose that~$|X|\geq 2$ and that the claim holds for all subsets of~$D_k$ smaller than~$X$.

Let~$\cT$ be a conflict-free spanning tree in~$\cH[X]$.
By the definition of a spanning tree, there are a flippable $\ell$-tuple~$\tau\in\cT$, a partition of~$X$ into non-empty subsets $X_1,\ldots,X_\ell$, and spanning trees $\cT_1,\ldots,\cT_\ell$ of $\cH[X_1],\ldots,\cH[X_\ell]$ (respectively) such that $\cT=\{\tau\}\cup\cT_1\cup\cdots\cup\cT_\ell$ and $\lvert\supp\tau\cap X_i\rvert=1$ for each~$i\in[\ell]$.
Since $\cT_1,\ldots,\cT_\ell\subset\cT$, the spanning trees $\cT_1,\ldots,\cT_\ell$ are conflict-free.
Therefore, by the induction hypothesis, the sets $S(X_1,\cT_1),\ldots,S(X_\ell,\cT_\ell)$ form Hamilton cycles in $G_k^+[X_1],\ldots,G_k^+[X_\ell]$.
Suppose $\tau=\{(x_1,m_1),\ldots,(x_\ell,m_\ell)\}$, where~$x_i\in X_i$ and~$m_i\in[2k]$ for~$i\in[\ell]$.
The unique common edge of~$W(\tau)$ with~$G_k^+[X_i]$ is the edge~$e(x_i,m_i)$, which belongs to~$S(X_i,\cT_i)$, as~$\cT$ is conflict-free, for~$i\in[\ell]$.
The set~$S(X,\cT)$ is the symmetric difference of $S(X_1,\cT_1),\ldots,S(X_\ell,\cT_\ell)$ and~$W(\tau)$, which is therefore a single cycle---a Hamilton cycle in~$G_k^+[X]$.

Now, suppose that~$\cH[X]$ has another conflict-free set~$\cU$ such that $S(X,\cT)=S(X,\cU)=:S$.
For each~$i\in[\ell]$, since the only edges in~$S$ that connect~$X_i$ with~$X\setminus X_i$ are those that precede and follow~$e(x_i,m_i)$ on~$W(\tau)$, these two edges along with~$e(x_i,m_i)$ belong to the same cycle witnessing some flippable tuple in~$\cU$.
It follows that the entire flipping cycle~$W(\tau)$ witnesses one of the flippable tuples in~$\cU$, which implies~$\tau\in\cU$.
The symmetric difference of~$S$ and~$W(\tau)$ is the disjoint union of $S(X_1,\cT_1),\ldots,S(X_\ell,\cT_\ell)$.
This implies that every flippable tuple in $\cU\setminus\{\tau\}$ is a flippable tuple on one of $X_1,\ldots,X_\ell$.
Therefore, we have $\cU=\{\tau\}\cup\cU_1\cup\cdots\cup\cU_\ell$, where~$\cU_i$ is a conflict-free set in~$\cH[X_i]$ such that $S(X_i,\cT_i)=S(X_i,\cU_i)$ for each~$i\in[\ell]$.
This and the induction hypothesis yield~$\cT_i=\cU_i$ for each~$i\in[\ell]$ and therefore~$\cT=\cU$.
\end{proof}

To apply Lemma~\ref{lem:reduction}, we need to define a flippability hypergraph on~$D_k$ that admits a conflict-free spanning tree.
In other words, we need to construct a sufficiently large set of flippable tuples.
Our construction works inductively and is based on the next lemma, which allows us to generate more flippable tuples from existing ones by prepending and appending certain bitstrings to them.
We introduce the following auxiliary notation for every flippable tuple $\tau=\{(x_1,m_1),\ldots,(x_\ell,m_\ell)\}$ on~$D_k$:
\begin{equation}
\label{eq:append-mirror-tau}
\begin{alignedat}{2}
u\tau v       &:= \{u(x_1,m_1)v,\ldots,u(x_\ell,m_\ell)v\} = \{(ux_1v,|u|+m_1),\ldots,(ux_\ell v,|u|+m_\ell)\}\\
              & \hspace{16em}\text{for any bitstrings }u\text{ and }v\text{ such that }uv\in D,\\
\revinv{\tau} &:= \{\revinv{(x_1,m_1)},\ldots,\revinv{(x_\ell,m_\ell)}\} = \{(\revinv{x_1},2k+1-m_1),\ldots,(\revinv{x_\ell},2k+1-m_\ell)\}.
\end{alignedat}
\end{equation}

\begin{lemma}
\label{lem:flip-closure}
If\/~$\tau$ is a flippable tuple, then
\begin{enumerate}
\item\label{item:flip-closure-i} $u\tau v$ is a flippable tuple for any bitstrings\/~$u$ and\/~$v$ such that\/~$uv\in D$ and\/~$|u|$ is even,
\item\label{item:flip-closure-ii} $u\revinv{\tau}v$ is a flippable tuple for any bitstrings\/~$u$ and\/~$v$ such that\/~$uv\in D$ and\/~$|u|$ is odd.
\end{enumerate}
\end{lemma}

The proofs of Lemma~\ref{lem:flip-closure} and of all subsequent lemmas stated in this section are deferred to Sections~\ref{sec:flip-lemmas} and~\ref{sec:spanning}.

We now specify the base case for our inductive construction of flippable tuples.
We found the following tuples on~$D_3$ and~$D_4$ with the help of a computer.
In fact, the computer search gave many more flippable tuples, and we carefully selected a subset that can be used to create a conflict-free spanning tree on~$D_k$.
The flippable tuples in this basic set~$\Phi$ of flippable tuples are called \emph{patterns}.
We let $\Phi:=\{\alpha(w)\mid w\in D\}\cup\{\beta,\gamma,\delta\}$, where
\begin{equation}
\label{eq:patterns}
\begin{alignedat}{2}
\alpha(w) &:= \{1w110\underline{0}0,1w1010\underline{0},1w\underline{1}0010\}, \qquad &
\gamma    &:= \{1\underline{1}001100,1101100\underline{0},11101\underline{0}00\}, \\ 
\beta     &:= \{11100\underline{0},1011\underline{0}0,\underline{1}01010\}, \qquad &
\delta    &:= \{11100\underline{0},1101\underline{0}0,10\underline{1}100,\underline{1}01010\}.
\end{alignedat}
\end{equation}
The Dyck path representation of these tuples is shown in Figure~\ref{fig:patterns}.
The next lemma asserts that these definitions indeed yield flippable tuples.

\begin{lemma}
\label{lem:flip-base}
Every pattern in\/~$\Phi$ defined by~\eqref{eq:patterns} is a flippable tuple.
\end{lemma}

\noindent
Figure~\ref{fig:P3} shows three flipping cycles that witness the patterns~$\alpha(\epsilon)$, $\beta$, and~$\delta$.

We use Lemmas~\ref{lem:flip-closure} and~\ref{lem:flip-base} to construct a set~$\Psi$ of flippable tuples.
Namely, we define
\begin{equation}
\label{eq:Psi}
\Psi:=\{u\varphi v\mid\varphi\in\Phi,\:uv\in D,\:\text{and }|u|\text{ is even}\}\cup\{u\revinv{\varphi}v\mid\varphi\in\Phi,\:uv\in D,\:\text{and }|u|\text{ is odd}\}.
\end{equation}
By Lemmas~\ref{lem:flip-closure} and~\ref{lem:flip-base}, every marked tuple in~$\Psi$ is flippable.
Observe that the set of flippable tuples~$\Psi$ is already closed with respect to the operation described in Lemma~\ref{lem:flip-closure}, that is,
\begin{equation}
\label{eq:Psi-closed}
\Psi=\{u\tau v\mid\tau\in\Psi,\:uv\in D,\:\text{and }|u|\text{ is even}\}\cup\{u\revinv{\tau}v\mid\tau\in\Psi,\:uv\in D,\:\text{and }|u|\text{ is odd}\}.
\end{equation}
Next, for each~$k\geq 2$, we define a set~$\Psi_k$ by extracting only the flippable tuples on~$D_k$ from~$\Psi$:
\begin{equation}
\label{eq:Psi_k}
\Psi_2:=\emptyset, \hspace{3em} \Psi_k:=\{\tau\in\Psi\mid\tau\text{ is a flippable tuple on }D_k\} \quad\text{for }k\geq 3.
\end{equation}
Finally, we define a flippability hypergraph $\cH_k:=(D_k,\Psi_k)$ for~$k\geq 2$.

\begin{lemma}
\label{lem:conflict-free}
For every\/~$k\geq 3$, the set\/~$\Psi_k$ of flippable tuples defined by~\eqref{eq:Psi_k} has the property that for any\/ $\tau_1,\tau_2\in\Psi_k$, if\/ $\supp\tau_1\cap\supp\tau_2=\{x\}$ where\/~$x\in D_k$, then the mark of\/~$x$ in\/~$\tau_1$ is different from the mark of\/~$x$ in\/~$\tau_2$.
In particular, every spanning tree of the hypergraph\/~$\cH_k$ is conflict-free.
\end{lemma}

In view of Lemmas~\ref{lem:reduction} and~\ref{lem:conflict-free}, it remains to prove that the hypergraph~$\cH_k$ has a spanning tree (many distinct spanning trees) to complete the proofs of Theorems~\ref{thm:odd} and~\ref{thm:count}.

\begin{lemma}
\label{lem:tree}
For every\/~$k\geq 3$, the hypergraph\/~$\cH_k$ has a spanning tree.
\end{lemma}

\begin{lemma}
\label{lem:tree-count}
For every\/~$k\geq 6$, the hypergraph\/~$\cH_k$ has at least\/~$2^{2^{k-6}}$ distinct spanning trees.
\end{lemma}

\begin{proof}[Proof of Theorem~\ref{thm:odd}]
Combine Lemma~\ref{lem:reduction}, Lemma~\ref{lem:conflict-free}, and Lemma~\ref{lem:tree}. 
\end{proof}

\begin{proof}[Proof of Theorem~\ref{thm:sparse}]
Combine Theorem~\ref{thm:odd} and~\cite[Theorem~1]{MR2836824}.
\end{proof}

\begin{proof}[Proof of Theorem~\ref{thm:count}]
Combine Lemma~\ref{lem:reduction}, Lemma~\ref{lem:conflict-free}, and Lemma~\ref{lem:tree-count}. 
\end{proof}

\section{Proofs of Lemmas~\ref{lem:paths}, \ref{lem:flip-closure}, \ref{lem:flip-base}, and~\ref{lem:conflict-free}}
\label{sec:flip-lemmas}

\begin{proof}[Proof of Lemma~\ref{lem:paths}]
We extend the mirroring notation to arbitrary bitstrings: for a bitstring $w=b_1b_2\cdots b_{2k}$, where $b_1,\ldots,b_{2k}\in\{0,1\}$, we define $\revinv{w}:=\ol{b_{2k}b_{2k-1}\cdots b_1}$.

We have $P(\epsilon)=(\epsilon)$ by definition.
When~$k\geq 1$ and~$x\in D_k$, the path $P(x)=(x_0,x_1,\ldots,x_{2k})$ can be described recursively as follows.
Let~$x=1u0v$ be the unique decomposition of~$x$ with~$u,v\in D$ that is used in~\eqref{eq:pi}.
Let $\ell=\frac{1}{2}|u|+1$, so that~$u\in D_{\ell-1}$ and~$v\in D_{k-\ell}$.
Then we have
\begin{equation}
\label{eq:paths}
\begin{alignedat}{3}
x_i&=1u0v &\quad &\text{for }i=0, \\
x_i&=1u_{i-1}1v &\quad &\text{for }i\in\{1,\ldots,2\ell-1\}, &\quad &\text{where }P(\revinv{u})=(\revinv{u_0},\revinv{u_1},\ldots,\revinv{u_{2\ell-2}}), \\
x_i&=0\ol{u}1v_{i-2\ell} &\quad &\text{for }i\in\{2\ell,\ldots,2k\}, &\quad &\text{where }P(v)=(v_0,v_1,\ldots,v_{2k-2\ell}).
\end{alignedat}
\end{equation}
Conversely, for every bitstring $y\in B_k$ with $k\geq 1$, exactly one of the following three cases holds:
\begin{enumerate}
\item\label{item:path-start} $y$ has a unique decomposition $y=1u0v$ with $u\in D_{\ell-1}$ and $v\in D_{k-\ell}$ where $\ell\in[k]$;
\item\label{item:path-rec-u} $y$ has a unique decomposition $y=1w1v$ with $\revinv{w}\in B_{\ell-1}$ and $v\in D_{k-\ell}$ where $\ell\in[k]$;
\item\label{item:path-rec-v} $y$ has a unique decomposition $y=0\ol{u}1w$ with $u\in D_{\ell-1}$ and $w\in B_{k-\ell}$ where $\ell\in[k]$.
\end{enumerate}

We prove that for every bitstring~$y\in B_k$ with~$k\geq 0$, there is exactly one pair~$(x,i)$ such that~$x\in D_k$, $i\in\{0,\ldots,2k\}$, and~$y=x_i$ given that $P(x)=(x_0,x_1,\ldots,x_{2k})$.
The proof goes by induction on~$k$.
The base case~$k=0$ trivial.
For the induction step, let~$y\in B_k$ with~$k\geq 1$, and suppose the statement holds for all bitstrings in $B_0\cup\cdots\cup B_{k-1}$.

We consider the three cases of the decomposition of~$y$ described in~\ref{item:path-start}--\ref{item:path-rec-v}.
In case~\ref{item:path-start}, it follows from~\eqref{eq:paths} that~$(1u0v,0)$ is the unique pair~$(x,i)$ with the required properties.
In case~\ref{item:path-rec-u}, we apply the induction hypothesis to~$\revinv{w}$ to get a unique pair~$(u,j)$ such that~$u\in D_{\ell-1}$, $j\in\{0,\ldots,2\ell-2\}$, and~$\revinv{w}=\revinv{u_j}$ (that is, $w=u_j$) given that $P(\revinv{u})=(\revinv{u_0},\revinv{u_1},\ldots,\revinv{u_{2\ell-2}})$; it follows from~\eqref{eq:paths} that~$(1u0v,j+1)$ is the unique pair~$(x,i)$ with the required properties.
Finally, in case~\ref{item:path-rec-v}, we apply the induction hypothesis to~$w$ to get a unique pair~$(v,j)$ such that~$v\in D_{k-\ell}$, $j\in\{0,\ldots,2k-2\ell\}$, and~$w=v_j$ given that $P(v)=(v_0,v_1,\ldots,v_{2k-2\ell})$; it follows from~\eqref{eq:paths} that~$(1u0v,2\ell+j)$ is the unique pair~$(x,i)$ with the required properties.
\end{proof}

\begin{proof}[Proof of Lemma~\ref{lem:flip-closure}]
We will prove the following three special cases of the statements claimed in the lemma: if $\tau$ is a flippable tuple, then
\begin{enumerate}
\setcounter{enumi}{2}
\item\label{item:flip-closure-iii} $u\tau$ is a flippable tuple for every~$u\in D$,
\item\label{item:flip-closure-iv} $\tau v$ is a flippable tuple for every~$v\in D$,
\item\label{item:flip-closure-v} $1\revinv{\tau}0$ is a flippable tuple.
\end{enumerate}
Statements \ref{item:flip-closure-i} and~\ref{item:flip-closure-ii} then follow by straightforward induction, because the operations described therein can be obtained by repeated application of the operations described in \ref{item:flip-closure-iii}--\ref{item:flip-closure-v}.

We will need the following simple observation.

\begin{claim}
For any $x,y\in D$, we have $\pi(xy)=\bigl(\pi(x),\:|x|+\pi(y)\bigr)$.
\end{claim}

The claim is proved by induction on~$|x|$.
If $x=\epsilon$, then the claim holds trivially.
Otherwise, let $x=1u0v$ be the unique decomposition of~$x$ such that $u,v\in D$.
By~\eqref{eq:pi} and by the induction hypothesis applied to~$vy$, we have
\begin{align*}
\pi(1u0vy) &= \bigl(|u|+2,\:(|u|+2)-\pi(\revinv{u}),\:1,\:(|u|+2)+\pi(vy)\bigr)\\
&= \bigl(|u|+2,\:(|u|+2)-\pi(\revinv{u}),\:1,\:(|u|+2)+\pi(v),\:(|u|+2+|v|)+\pi(y)\bigr)\\
&=\bigl(\pi(1u0v),\:(|u|+2+|v|)+\pi(y)\bigr),
\end{align*}
which proves the claim.

Let $\tau=\{(x_1,m_1),\ldots,(x_\ell,m_\ell)\}$ be a flippable tuple in~$D_k$, where $k\geq 1$, and let $C=(y_1,\ldots,y_{2\ell})$ be a flipping cycle of length~$2\ell$ that witnesses~$\tau$.

For the proof of~\ref{item:flip-closure-iii}, let $u\in D$.
The claim implies that for each~$i\in[\ell]$, the final part of the path~$P(ux_i)$ looks as follows:
\[P(ux_i)=(\ldots,\ol{u}z_0,\ol{u}z_1,\ldots,\ol{u}z_{2k}),\quad\text{where }P(x_i)=(z_0,z_1,\ldots,z_{2k}).\]
Therefore, $(\ol{u}y_1,\ldots,\ol{u}y_{2\ell})$ is a flipping cycle that witnesses~$u\tau$.
This proves~\ref{item:flip-closure-iii}.

For the proof of~\ref{item:flip-closure-iv}, let $v\in D$.
The claim implies that for each~$i\in[\ell]$, the initial part of the path~$P(x_iv)$ looks as follows:
\[P(x_iv)=(z_0v,z_1v,\ldots,z_{2k}v,\ldots),\quad\text{where }P(x_i)=(z_0,z_1,\ldots,z_{2k}).\]
Therefore, $(y_1v,\ldots,y_{2\ell}v)$ is a flipping cycle that witnesses~$\tau v$.
This proves~\ref{item:flip-closure-iv}.

Finally, we prove~\ref{item:flip-closure-v}.
For each~$i\in[\ell]$, by~\eqref{eq:pi}, we have $\pi(1x_i0)=\bigl(2k+2,\:(2k+2)-\pi(\revinv{x_i}),\:1\bigr)$, which implies that
\[P(1x_i0)=(1x_i0,1z_01v,1z_11v,\ldots,1z_{2k}1v,0\ol{x_i}1),\quad\text{where }P(\revinv{x_i})=(\revinv{z_0},\revinv{z_1},\ldots,\revinv{z_{2k}}).\]
Therefore, $(1\revinv{y_1}1,\ldots,1\revinv{y_{2\ell}}1)$ is a flipping cycle that witnesses $1\revinv{\tau}0$.
This proves~\ref{item:flip-closure-v}.
\end{proof}

\begin{proof}[Proof of Lemma~\ref{lem:flip-base}]
Consider the sequences
\begin{align}
\label{eq:flipping-cycles}
\begin{split}
C_{\alpha(w)} &:= (1w00101,1w00111,1w00110,1w10110,1w10100,1w10101) \quad \text{for }w\in D, \\
C_{\beta}     &:= (111000,111001,011001,011011,011010,111010), \\
C_{\gamma}    &:= (11011100,10011100,10011101,10011001,11011001,11011000), \\
C_{\delta}    &:= (111000,111001,110001,110011,010011,011011,011010,111010).
\end{split}
\end{align}
It is easy to verify that each of these sequences is a cycle in~$G_k$ for the appropriate value of~$k$.

We claim that for each pattern~$\varphi \in \Phi$, the cycle~$C_\varphi$ from~\eqref{eq:flipping-cycles} is a flipping cycle that witnesses~$\varphi$.
For $\varphi \in \{\beta,\delta\}$ this can be verified directly from Figure~\ref{fig:P3}, as follows.
The bitstrings on the cycles~$C_\beta$ and~$C_\delta$ are indicated in the figure by dotted and solid frames, respectively.
Both cycles have exactly one edge in common with each of the paths that start at the respective vertices in the tuples~$\beta$ and~$\delta$.
Furthermore, the bits flipped along the indicated edges are precisely those that are marked in~$\beta$ and~$\delta$.

In the same way, the claim can be verified for~$\varphi=\alpha(w)$ when~$w=\epsilon$.
The bitstrings on the cycle~$C_{\alpha(\epsilon)}$ are indicated in Figure~\ref{fig:P3} by dashed frames.
For general~$w\in D$, by the definition~\eqref{eq:pi}, the initial parts of the paths in~$\cP_{3+|w|/2}$ that start at the members of~$\alpha(w)$ look as follows:
\begin{equation}
\label{eq:alpha-paths}
\begin{alignedat}{3}
P(1w11000) &= (1w11000,1w11001,1w01001,1w01101,1w00101,1w00111,\ldots), \\
P(1w10100) &= (1w10100,1w10101,\ldots), \\
P(1w10010) &= (1w10010,1w10110,1w00110,\ldots).
\end{alignedat}
\end{equation}
For~$w=\epsilon$, these paths are exactly the same as~$P(x_1)$, $P(x_2)$, and~$P(x_3)$ in Figure~\ref{fig:P3}.
The $6$-cycle~$C_{\alpha(w)}$ defined by~\eqref{eq:flipping-cycles} intersects every path from~\eqref{eq:alpha-paths} exactly at the last edge explicitly shown in~\eqref{eq:alpha-paths}.
Furthermore, the bits flipped along the intersection edges are exactly the marked bits of the members of~$\alpha(w)$ as defined by~\eqref{eq:patterns}.

Finally, we consider the case~$\varphi=\gamma$.
The initial parts of the paths in~$\cP_4$ that start at the members of~$\gamma$ look as follows:
\begin{equation}
\label{eq:gamma-paths}
\begin{alignedat}{3}
P(11001100) &= (11001100,11011100,10011100,\ldots), \\
P(11011000) &= (11011000,11011001,\ldots), \\
P(11101000) &= (11101000,11101001,10101001,10111001,10011001,10011101,\ldots).
\end{alignedat}
\end{equation}
The cycle~$C_\gamma$ defined by~\eqref{eq:flipping-cycles} intersects every path from~\eqref{eq:gamma-paths} exactly at the last edge explicitly shown in~\eqref{eq:gamma-paths}, and the bits flipped along the intersection edges are exactly the marked bits of the members of~$\gamma$ as defined by~\eqref{eq:patterns}.
\end{proof}

\begin{figure}[p]
\includegraphics[scale=0.916]{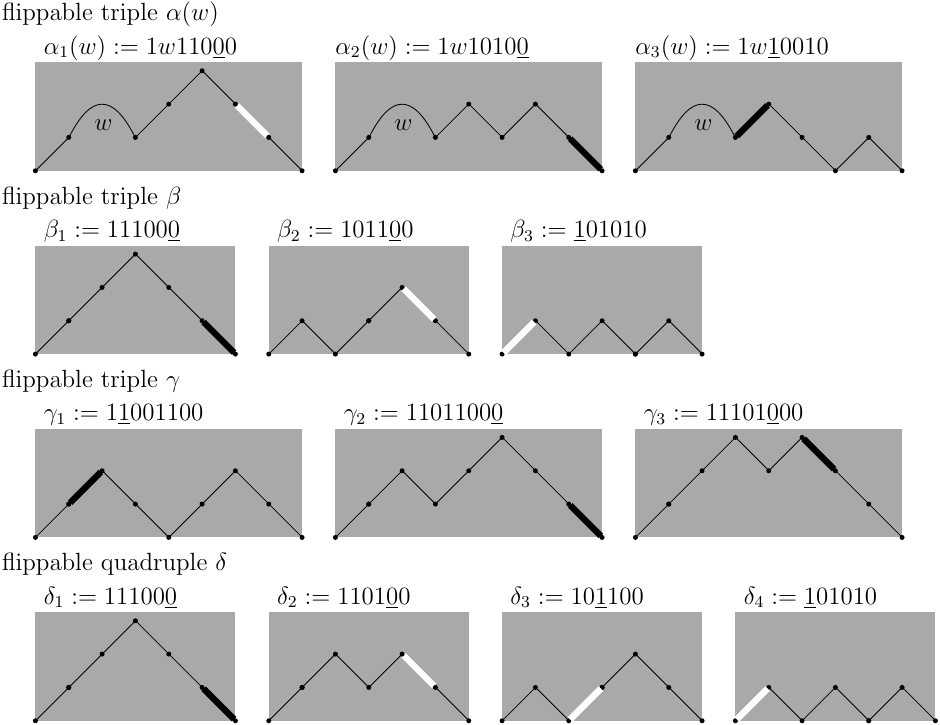}
\caption{Dyck path representation of the patterns~$\Phi$ defined by~\eqref{eq:patterns}.
The steps of the Dyck paths that represent the marked bits are highlighted white at odd positions and black at even positions.}
\label{fig:patterns}
\end{figure}

\begin{figure}[p]
\includegraphics[scale=0.916]{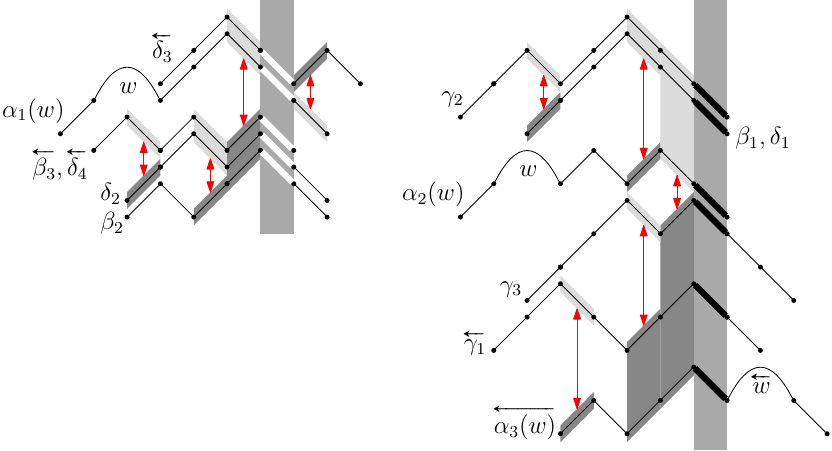}
\caption{Illustration for the proof of Lemma~\ref{lem:conflict-free}.
The double arrows indicate mismatches between the Dyck paths aligned at the marked down-steps.}
\label{fig:conflict}
\end{figure}

\begin{proof}[Proof of Lemma~\ref{lem:conflict-free}]
Let $\revinv{\Phi}:=\{\revinv{\varphi}\mid\varphi\in\Phi\}$.
Let~$\tau_1$ and~$\tau_2$ be flippable tuples in~$\Psi_k$ such that $\supp\tau_1\cap\supp\tau_2=\{x\}$, where~$x\in D_k$.
Let~$i\in\{1,2\}$.
By~\eqref{eq:Psi}, we have $\tau_i=u_i\varphi_iv_i$, where~$u_i$ and~$v_i$ are some bitstrings such that~$u_iv_i\in D$ and
\begin{equation}
\label{eq:windows}
\varphi_i\in\Phi\quad\text{if }|u_i|\text{ is even}, \hspace{4em} \varphi_i\in\revinv{\Phi}\quad\text{if }|u_i|\text{ is odd}.
\end{equation}
Since~$x\in\supp\tau_i$, we have~$x=u_ix_iv_i$ for some Dyck word~$x_i\in\supp\varphi_i$.
Let~$m_i$ be the mark of~$x_i$ in~$\varphi_i$.
It follows that~$|u_i|+m_i$ is the mark of~$x$ in~$\tau_i$.

Suppose for the sake of contradiction that the mark of~$x$ in~$\tau_1$ is the same as the mark of~$x$ in~$\tau_2$, that is, $|u_1|+m_1=|u_2|+m_2=:m$.
Let $p:=\max(1-m_1,1-m_2)$ and $q:=\min(|x_1|-m_1,|x_2|-m_2)$.
Thus we have $p\leq 0\leq q$.
The fact that $u_1x_1v_1=u_2x_2v_2$ and $|u_1|+m_1=|u_2|+m_2$ implies that for each $i\in\{p,p+1,\ldots,q\}$, the $(m_1+i)$th bit of~$x_1$ is equal to the $(m_2+i)$th bit of~$x_2$.
We claim that this is possible only when $(x_1,m_1)=(x_2,m_2)$.
The proof of this claim involves consideration of all possible cases of~$\varphi_1$ and~$\varphi_2$ satisfying~\eqref{eq:windows} and all possible cases of marked Dyck words $(x_1,m_1)\in\varphi_1$ and $(x_2,m_2)\in\varphi_2$.
To avoid tedious case distinctions, we propose a visual argument using the Dyck path representation of~$x$.
Figure~\ref{fig:patterns} presents the Dyck paths of the members of the patterns in~$\Phi$ in which the steps representing the marked bits have been marked white or black according to the following \emph{marking rule}: steps at odd positions are marked white, and steps at even positions are marked black.

Consider the Dyck path~$x$ in which the $m$th step has been marked white or black according to this rule.
Let~$i\in\{1,2\}$.
The common part of such a marked Dyck path~$x$ with the infinite vertical strip $[|u_i|,|u_ix_i|]\times\mathbb{R}$ is a translated copy of the Dyck path~$x_i$ in which the $m_i$th step has been marked according to the marking rule above.
That is, it has the form of one of the marked Dyck paths in Figure~\ref{fig:patterns} if~$|u_i|$ is even or the form of the mirror image of one of the marked Dyck paths in Figure~\ref{fig:patterns} if~$|u_i|$ is odd.
Note that the mirroring does \emph{not} change the mark colors---even though the parity of the relative position of the mark from the left within the Dyck path changes, this is compensated by the opposite parity of~$|u_i|$.
The two translated marked Dyck paths~$x_1$ and~$x_2$ must coincide on the common part of the two vertical strips.

Figure~\ref{fig:conflict} presents every Dyck path from Figure~\ref{fig:patterns} or its mirror image, so that the marked step is a down-step.
It shows them aligned horizontally with respect to the marked down-steps, separately for each color.
It can be checked in the figure that no two of these marked Dyck paths can coincide on the common part of the two vertical strips considered in the argument above, unless they are the same marked Dyck path.
Specifically, a mismatch between any two distinct Dyck paths is indicated by a double arrow in the figure.
The situation when the marked step is an up-step is analogous, by symmetry.

We have argued that $(x_1,m_1)=(x_2,m_2)$.
This is possible only when $\{\varphi_1,\varphi_2\}=\{\beta,\delta\}$, $\{\varphi_1,\varphi_2\}=\{\revinv{\beta},\revinv{\delta}\}$, or otherwise $\varphi_1=\varphi_2$.
In any case, we have $\lvert\supp\varphi_1\cap\supp\varphi_2\rvert\geq 3$.
The assumption that $|u_1|+m_1=|u_2|+m_2$ implies $|u_1|=|u_2|$, which implies~$u_1=u_2$ and~$v_1=v_2$.
Therefore, for each $y\in\supp\varphi_1\cap\supp\varphi_2$, we have $u_1yv_1=u_2yv_2\in\supp u_1\varphi_1v_1\cap\supp u_2\varphi_2v_2=\supp\tau_1\cap\supp\tau_2$.
It follows that $\lvert\supp\tau_1\cap\supp\tau_2\rvert\geq 3$, which contradicts the assumption that $\supp\tau_1\cap\supp\tau_2=\{x\}$.

The second statement of Lemma~\ref{lem:conflict-free} is an immediate consequence of the first statement and the definition of a conflict-free spanning tree.
\end{proof}

\section{Proofs of Lemmas~\ref{lem:tree} and~\ref{lem:tree-count}}
\label{sec:spanning}

Before proceeding with the proofs, we generalize the notation~\eqref{eq:append-mirror-x} and~\eqref{eq:append-mirror-tau}.
For every~$k\geq 2$ and every set~$X\subseteq D_k$, we define
\begin{align*}
uXv        &:= \{uxv\mid x\in X\} \quad \text{for any bitstrings }u\text{ and }v\text{ such that }uv\in D,\\
\revinv{X} &:= \{\revinv{x}\mid x\in X\}.
\end{align*}
Similarly, for every~$k\geq 2$ and every set~$\cX$ of flippable tuples on~$D_k$, we define
\begin{align*}
u\cX v       &:= \{u\tau v\mid\tau\in\cX\} \quad \text{for any bitstrings }u\text{ and }v\text{ such that }uv\in D,\\
\revinv{\cX} &:= \{\revinv{\tau}\mid\tau\in\cX\}.
\end{align*}
As a direct consequence of the definitions above and~\eqref{eq:Psi-closed}, if~$X\subseteq D_k$ and~$\cX$ is a spanning tree of~$\cH_k[X]$, then the following holds for any bitstrings~$u$ and~$v$ such that~$uv\in D$:
\begin{enumerate}
\item if~$|u|$ is even, then~$u\cX v$ is a spanning tree of $\cH_{k+|uv|/2}[uXv]$;
\item if~$|u|$ is odd, then~$u\revinv{\cX}v$ is a spanning tree of $\cH_{k+|uv|/2}[u\revinv{X}v]$.
\end{enumerate}
We will use this property extensively in the proofs below.

\begin{figure}[t]
\includegraphics[scale=0.91]{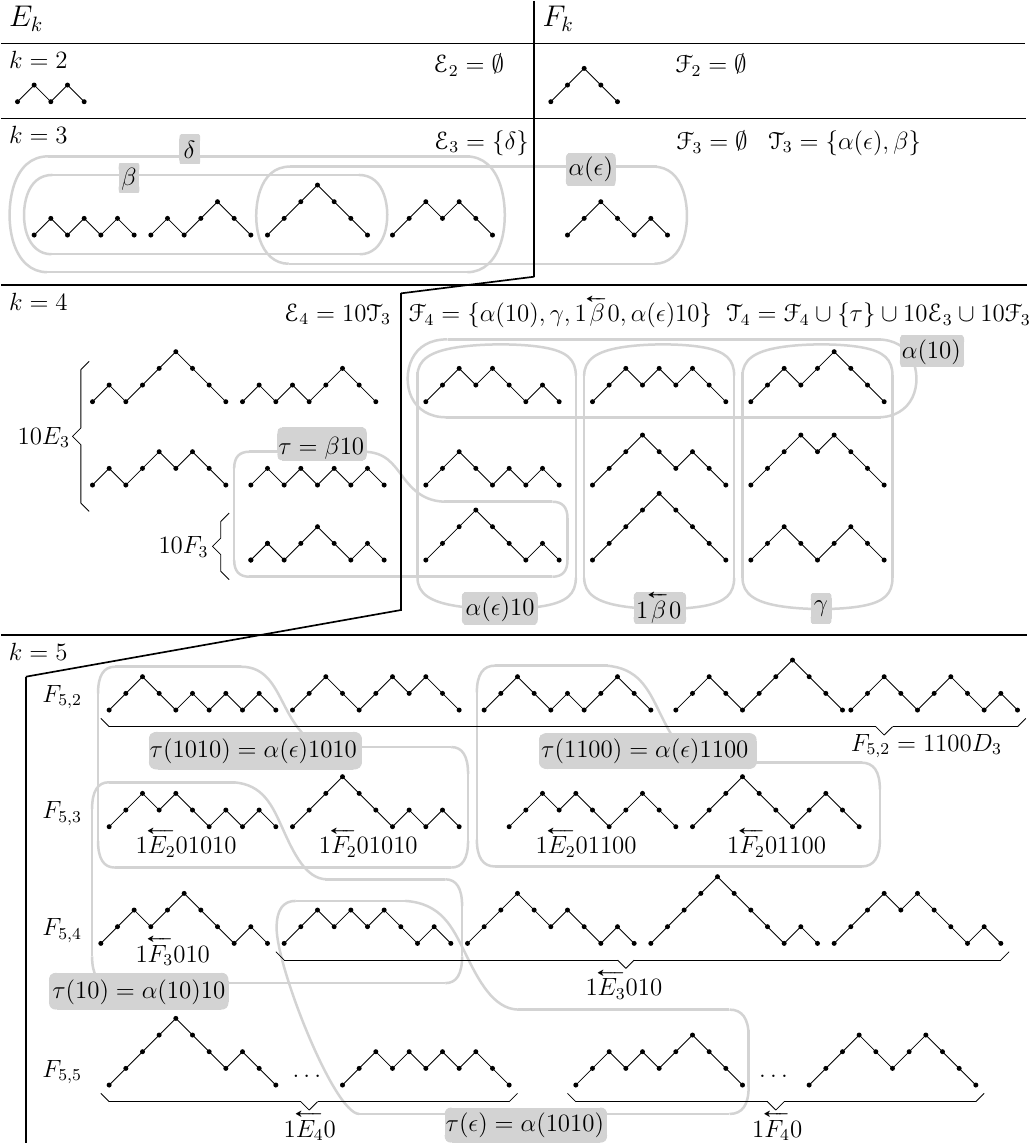}
\caption{Illustration of the proof of Lemma~\ref{lem:tree}.
The inductive construction of~$\cT_k$ is shown for~$k=4$, and the inductive construction of~$\cF_k$ is shown for~$k=5$.}
\label{fig:tree}
\end{figure}

\begin{proof}[Proof of Lemma~\ref{lem:tree}]
For the reader's convenience, this proof is illustrated in Figure~\ref{fig:tree}.

For~$k\geq 2$, we partition the set of Dyck words~$D_k$ into two sets~$E_k$ and~$F_k$ as follows:
\begin{equation}
\label{eq:E/F-definition}
\begin{alignedat}{4}
E_2 &:= \{1010\}, \hspace{2.5em} & E_3 &:= D_3\setminus\{110010\}, \hspace{2.5em} & E_k &:= 10D_{k-1}              \quad &&\text{for } k\geq 4,\\
F_2 &:= \{1100\}, \hspace{2.5em} & F_3 &:= \{110010\},             \hspace{2.5em} & F_k &:= D_k\setminus 10D_{k-1} \quad &&\text{for } k\geq 4.
\end{alignedat}
\end{equation}
In particular, we have the following, for~$k\geq 2$:
\begin{equation}
\label{eq:E/F-members}
1010(10)^{k-2}\in E_k, \hspace{4em} 1100(10)^{k-2}\in F_k.
\end{equation}

We prove the following more general statement, which directly implies the lemma:

\begin{claim}
There are a spanning tree~$\cT_k$ of~$\cH_k$ for~$k\geq 3$, a spanning tree~$\cE_k$ of~$\cH_k[E_k]$ for~$k\geq 2$, and a spanning tree~$\cF_k$ of~$\cH_k[F_k]$ for~$k\geq 2$.
\end{claim}

We prove the claim by induction on~$k$.

For~$k=2$, we let~$\cE_2:=\emptyset$ and~$\cF_2:=\emptyset$, which trivially satisfy the conditions for a spanning tree, as $|E_2|=|F_2|=1$.

For~$k=3$, we let~$\cE_3:=\{\delta\}$ and~$\cF_3:=\emptyset$, which satisfy the conditions for a spanning tree, as $\supp\delta=E_3$ and~$|F_3|=1$.
We also let $\cT_3:=\{\alpha(\epsilon),\beta\}$, which satisfies the conditions for a spanning tree, because $\supp\alpha(\epsilon)\cup\supp\beta=D_3$ and $\lvert\supp\alpha(\epsilon)\cap\supp\beta\rvert=1$.

For~$k=4$, we let $\cF_4:=\{\tau_1,\tau_2,\tau_3,\tau_4\}$, where
\begin{align*}
\tau_1 &:= \{110110\underline{0}0,1101010\underline{0},110\underline{1}0010\} = \alpha(10),\\
\tau_2 &:= \{1\underline{1}001100,1101100\underline{0},11101\underline{0}00\} = \gamma,\\
\tau_3 &:= \{1\underline{1}110000,11\underline{1}00100,110101\underline{0}0\} = 1\revinv{\beta}0,\\
\tau_4 &:= \{1110\underline{0}010,11010\underline{0}10,1\underline{1}001010\} = \alpha(\epsilon)10.
\end{align*}
They belong to~$\Psi_k$ by construction.
Moreover, we have $\supp\tau_1\cup\supp\tau_2\cup\supp\tau_3\cup\supp\tau_4=F_4$, $\lvert\supp\tau_1\cap\supp\tau_i\rvert=1$ for every~$i\in\{2,3,4\}$, and $\supp\tau_i\cap\supp\tau_j=\emptyset$ for any distinct $i,j\in\{2,3,4\}$.
This shows that~$\cF_4$ is indeed a spanning tree of~$\cH_4[F_4]$.

Finally, we proceed by induction on~$k$ to construct~$\cT_k$ and~$\cE_k$ for~$k\geq 4$, and~$\cF_k$ for~$k\geq 5$.
The construction of~$\cE_k$ makes use of~$\cT_{k-1}$.
The construction of~$\cF_k$ (for~$k\geq 5$) makes use of $\cE_2,\ldots,\cE_{k-1}$, $\cF_2,\ldots,\cF_{k-1}$, and~$\cT_{k-2}$ (this is why the case~$k=4$ needs to be considered separately, as~$\cH_2$ has no spanning tree).
Finally, the construction of~$\cT_k$ makes use of~$\cF_k$,~$\cE_{k-1}$, and~$\cF_{k-1}$.

Constructing~$\cE_k$ from~$\cT_{k-1}$ for~$k\geq 4$ is straightforward: since $E_k=10D_{k-1}$, it suffices to take $\cE_k:=10\cT_{k-1}$, which is a spanning tree of~$\cH_k[E_k]$.

Now, we show how to construct~$\cF_k$ from $\cE_2,\ldots,\cE_{k-1}$, $\cF_2,\ldots,\cF_{k-1}$, and~$\cT_{k-2}$ for~$k\geq 5$.
For $2\leq j\leq k$, let $F_{k,j}:=\bigcup_{i=2}^j\{1\revinv{u}0v\mid u\in D_{i-1}$ and $v\in D_{k-i}\}$.
Since $F_k=F_{k,k}$, the following statement, which we prove by auxiliary induction on~$j$, directly implies the existence of a spanning tree of~$\cH_k[F_k]$:

\begin{claim}
For~$2\leq j\leq k$, there is a spanning tree~$\cF_{k,j}$ of~$\cH_k[F_{k,j}]$.
\end{claim}

For~$j=2$, we have $F_{k,2}=\{1100v\mid v\in D_{k-2}\}=1100D_{k-2}$, so we let $\cF_{k,2}:=1100\cT_{k-2}$.
Now, suppose $3\leq j\leq k$.
The fact that~$E_{j-1}$ and~$F_{j-1}$ form a partition of~$D_{j-1}$ yields the following partition of the set~$F_{k,j}$:
\begin{equation}
\label{eq:F-partition}
\begin{split}
F_{k,j} &= F_{k,j-1}\cup\{1\revinv{u}0v\mid u\in D_{j-1}\text{ and }v\in D_{k-j}\}\\
&= F_{k,j-1}\cup\!\!\bigcup_{v\in D_{k-j}}\!\!1\revinv{D_{j-1}}0v\\
&= F_{k,j-1}\cup\!\!\bigcup_{v\in D_{k-j}}\!\!\bigl(1\revinv{E_{j-1}}0v\cup 1\revinv{F_{j-1}}0v\bigr).
\end{split}
\end{equation}
For every~$v\in D_{k-j}$, consider the following marked triple on~$F_{k,j}$:
\begin{equation}
\label{eq:tau(v)}
\tau(v):=\{1(10)^{j-3}110\underline{0}0v,1(10)^{j-3}1010\underline{0}v,1(10)^{j-3}\underline{1}0010v\}=\alpha((10)^{j-3})v.
\end{equation}
It belongs to~$\Psi_k$ by construction.
We have
\begin{equation*}
1(10)^{j-3}11000v\in 1\revinv{F_{j-1}}0v,\quad 1(10)^{j-3}10100v\in 1\revinv{E_{j-1}}0v,\quad 1(10)^{j-3}10010v\in F_{k,j-1},
\end{equation*}
where the first two memberships follow from~\eqref{eq:E/F-members}.
We take the spanning trees of the subhypergraphs of~$\cH_k$ induced by the sets of the partition of~$F_{k,j}$ given by~\eqref{eq:F-partition} and connect them into a single spanning tree of~$\cH_k[F_{k,j}]$ using the triples~$\tau(v)$ for all~$v\in D_{k-j}$.
That is, we let
\begin{equation*}
\cF_{k,j}:=\cF_{k,j-1}\cup\!\!\bigcup_{v\in D_{k-j}}\!\!\bigl(\{\tau(v)\}\cup 1\revinv{\cE_{j-1}}0v\cup 1\revinv{\cF_{j-1}}0v\bigr),
\end{equation*}
which is a spanning tree of~$\cH_k[F_{k,j}]$.

Finally, we show how to construct~$\cT_k$ from~$\cF_k$, $\cE_{k-1}$ and~$\cF_{k-1}$.
Consider the following marked triple on~$D_k$:
\begin{equation*}
\tau:=\{11100\underline{0}(10)^{k-3},1011\underline{0}0(10)^{k-3},\underline{1}01010(10)^{k-3}\bigr\}=\beta(10)^{k-3}.
\end{equation*}
It belongs to~$\Psi_k$ by construction.
We have
\begin{equation*}
111000(10)^{k-3}\in F_k,\qquad 101100(10)^{k-3}\in 10E_{k-1},\qquad 101010(10)^{k-3}\in 10F_{k-1},
\end{equation*}
where the first membership is by the definition of~$F_k$ for~$k\geq 4$ and the other two follow from~\eqref{eq:E/F-members}.
The sets~$F_k$, $10E_{k-1}$, and~$10F_{k-1}$ form a partition of~$D_k$.
We take the spanning trees of the subhypergraphs induced by these partition sets and connect them into a single spanning tree of~$\cH_k$ using the triple~$\tau$.
That is, we let $\cT_k:=\cF_k\cup\{\tau\}\cup 10\cE_{k-1}\cup 10\cF_{k-1}$, which is a spanning tree of~$\cH_k$.
\end{proof}

\begin{proof}[Proof of Lemma~\ref{lem:tree-count}]
The proof proceeds along the same lines as the proof of Lemma~\ref{lem:tree}, so we only highlight the differences.

Apart from the partition of~$D_4$ into two sets~$E_4$ and~$F_4$ defined by~\eqref{eq:E/F-definition}, we will use another one---a partition into sets~$E'_4$ and~$F'_4$ defined as follows:
\begin{equation*}
E'_4:=D_4\setminus\{11001100\}, \hspace{5em} F'_4:=\{11001100\}.
\end{equation*}
It has the following property analogous to~\eqref{eq:E/F-members}:
\begin{equation}
\label{eq:E'/F'-members}
10101100\in E'_4, \hspace{6em} 11001100\in F'_4.
\end{equation}
We define spanning trees $\cE'_4:=\{\tau'_1,\tau'_2,\tau'_3,\tau'_4,\tau'_5\}$ of~$\cH_4[E'_4]$ and~$\cF'_4:=\emptyset$ of~$\cH_4[F'_4]$, where
\begin{alignat*}{2}
\tau'_1 &:= \{1\underline{1}110000,11\underline{1}01000,1110\underline{0}100,110101\underline{0}0\} && = \smash[t]{1\revinv{\delta}0},\\
\tau'_2 &:= \{110110\underline{0}0,1101010\underline{0},110\underline{1}0010\}                      && = \alpha(10),\\
\tau'_3 &:= \{1110\underline{0}010,11010\underline{0}10,1\underline{1}001010\}                      && = \alpha(\epsilon)10,\\
\tau'_4 &:= \{11100\underline{0}10,1011\underline{0}010,\underline{1}0101010\}                      && = \beta 10,\\
\tau'_5 &:= \{1011100\underline{0},101101\underline{0}0,1010\underline{1}100,10\underline{1}01010\} && = 10\delta.
\end{alignat*}
We have $\supp\tau'_1\cup\cdots\cup\supp\tau'_5=E'_4$, $\lvert\supp\tau'_{i-1}\cap\supp\tau'_i\rvert=1$ for $i\in\{2,3,4,5\}$, and $\supp\tau'_i\cap\supp\tau'_j=\emptyset$ whenever $|i-j|\geq 2$, which shows that~$\cE'_4$ is indeed a spanning tree of~$\cH_4[E'_4]$, and~$\cF'_4$ is a spanning tree of~$\cH_4[F'_4]$ because~$|F'_4|=1$.

To obtain many spanning trees of~$\cH_k$ for~$k\geq 6$, we proceed as in the proof of Lemma~\ref{lem:tree} except that we introduce variants to the construction of the spanning tree~$\cF_{k,5}$ of~$\cH_k[F_{k,5}]$.
Consider an arbitrary partition of~$D_{k-5}$ into two sets~$X$ and~$Y$.
The fact that~$E_4$ and~$F_4$ as well as~$E'_4$ and~$F'_4$ form partitions of~$D_4$ yields the following partition of~$F_{k,5}$ analogous to~\eqref{eq:F-partition}:
\begin{equation}
\label{eq:F-partition'}
\begin{split}
F_{k,5} &= F_{k,4}\cup\{1\revinv{u}0v\mid u\in D_4\text{ and }v\in D_{k-5}\}\\
&= F_{k,4}\cup\!\!\bigcup_{v\in D_{k-5}}\!\!1\revinv{D_4}0v\\
&= F_{k,4}\cup\bigcup_{v\in X}\bigl(1\revinv{E_4}0v\cup 1\revinv{F_4}0v\bigr)\cup\bigcup_{v\in Y}\bigl(1\revinv{E'_4}0v\cup 1\revinv{F'_4}0v\bigr).
\end{split}
\end{equation}
Consider the following triples on~$F_{k,5}$, where the first one is a special case of~\eqref{eq:tau(v)}:
\begin{alignat*}{3}
\tau(v)  &:= \{11010110\underline{0}0v,110101010\underline{0}v,11010\underline{1}0010v\} && = \alpha(1010)v \quad && \text{for }v\in X,\\
\tau'(v) &:= \{11100110\underline{0}0v,111001010\underline{0}v,11100\underline{1}0010v\} && = \alpha(1100)v \quad && \text{for }v\in Y.
\end{alignat*}
They belong to~$\Psi_k$ by construction.
We have
\begin{alignat*}{4}
1101011000v &\in 1\revinv{F_4}0v,  \quad & 1101010100v &\in 1\revinv{E_4}0v,  \quad & 1101010010v &\in F_{k,4}, \quad && \text{for }v\in X,\\
1110011000v &\in 1\revinv{F'_4}0v, \quad & 1110010100v &\in 1\revinv{E'_4}0v, \quad & 1110010010v &\in F_{k,4}, \quad && \text{for }v\in Y,
\end{alignat*}
where the first two memberships follow from~\eqref{eq:E/F-members} and~\eqref{eq:E'/F'-members}, respectively.
We take the spanning trees of the subhypergraphs of~$\cH_k$ induced by the sets of the partition of~$F_{k,5}$ given by~\eqref{eq:F-partition'} and connect them into a single spanning tree of~$\cH_k[F_{k,5}]$ using the triples~$\tau(v)$ for all~$v\in X$ and the triples~$\tau'(v)$ for all~$v\in Y$.
That is, we let
\begin{equation*}
\cF_{k,5}:=\cF_{k,4}\cup\bigcup_{v\in X}\bigl(\{\tau(v)\}\cup 1\revinv{\cE_4}0v\cup 1\revinv{\cF_4}0v\bigr)\cup\bigcup_{v\in Y}\bigl(\{\tau'(v)\}\cup 1\revinv{\cE'_4}0v\cup 1\revinv{\cF'_4}0v\bigr),
\end{equation*}
which is a spanning tree of~$\cH_k[F_{k,5}]$.
Then, we continue with the constructions of~$\cF_k$ and~$\cT_k$ exactly as in the proof of Lemma~\ref{lem:tree}.

Clearly, distinct choices of the partition of~$D_{k-5}$ into two sets~$X$ and~$Y$ in the procedure above give rise to distinct spanning trees~$\cF_{k,5}$ of~$\cH_k[F_{k,5}]$, which consequently give rise to distinct spanning trees~$\cF_k$ of~$\cH_k[F_k]$ and~$\cT_k$ of~$\cH_k$.
Since $|D_{k-5}|=\frac{1}{k-4}\binom{2k-10}{k-5}\geq 2^{k-6}$ for~$k\geq 6$, there are at least~$2^{2^{k-6}}$ distinct partitions of~$D_{k-5}$ into sets~$X$ and~$Y$, which give rise to at least~$2^{2^{k-6}}$ distinct spanning trees of~$\cH_k$.
\end{proof}

\section{Alternative proof of the middle levels conjecture}
\label{sec:mlc}

As we explained in Section~\ref{sec:Hnk}, Theorem~\ref{thm:odd} implies that the bipartite Kneser graph $H(2k+1,k)$ has a Hamilton path for every~$k\geq 1$.
We proceed to prove that it even has a Hamilton cycle, yielding an alternative proof of the middle levels conjecture, first proved in~\cite{MR3483129}.
Plugging in Theorem~\ref{thm:count} instead of Theorem~\ref{thm:odd}, we obtain an alternative proof of the fact that $H(2k+1,k)$ contains double-exponentially many distinct Hamilton cycles, also first proved in~\cite{MR3483129}.

\begin{theorem}
For every integer\/~$k\geq 1$, the bipartite Kneser graph\/~$H(2k+1,k)$ has a Hamilton cycle.
For every integer\/~$k\geq 6$, the bipartite Kneser graph\/~$H(2k+1,k)$ has at least\/ $2^{2^{k-6}}$ distinct Hamilton cycles.
\end{theorem}

\begin{proof}
For any graph~$G$ whose vertices are bitstrings, we let $Gx$ denote the graph obtained by appending a bitstring~$x$ to all vertices of~$G$, and we let $\ol{G}$ denote the graph obtained from~$G$ by taking the complement of each vertex.
The graph~$H(2k+1,k)$ is isomorphic to the graph~$M_k$ obtained as the disjoint union of~$G_k0$ and~$\ol{G_k}1$ plus the \emph{matching edges} $\{x0,x1\}$ with $x\in B_k^0$.
The isomorphism is given by interpreting all bitstrings of length~$2k+1$ as characteristic vectors of $k$-element and $(k+1)$-element subsets of~$[2k+1]$.

For~$k=1$ and~$k=2$, the theorem can be verified directly.
For~$k\geq 3$, Theorem~\ref{thm:odd} yields a Hamilton cycle in the graph~$G_k^+$.
Moreover, it follows from the proof of Theorem~\ref{thm:odd} that this Hamilton cycle is obtained as the symmetric difference of the cycle factor~$\cC_k$ in~$G_k^+$ with some set of flipping cycles.
The cycle factor~$\cC_k$ is obtained from a collection of paths~$\cP_k$ in~$G_k$ by adding the edges $\{x,\ol{x}\}$ with $x\in D_k$, where the path $P(x)\in\cP_k$ connects $x$ and~$\ol{x}$.
The flipping cycles are cycles in~$G_k$.
Therefore, the resulting Hamilton cycle~$C$ in~$G_k^+$ contains some edges of~$G_k$ plus the edges $\{x,\ol{x}\}$ with $x\in D_k$.
We remove the latter edges from~$C$, thus obtaining a collection of paths~$\cQ$ in~$G_k$.
Then, to obtain a Hamilton cycle in~$M_k$, we take the paths in $\cQ 0$ and~$\ol{\cP_k}1$ and add the matching edges $\{x0,x1\}$ and~$\{\ol{x}0,\ol{x}1\}$ with $x\in D_k$.
This is indeed a Hamilton cycle in~$M_k$, obtained from the Hamilton cycle $C0$ in~$G_k^+0$ by removing every edge of the form $\{x0,\ol{x}0\}$ (with $x\in D_k$) and replacing it by the path going from~$\ol{x}0$ to~$\ol{x}1$, then along $\ol{P(x)}1$ from~$\ol{x}1$ to~$x1$, and then to~$x0$.

For~$k\geq 6$, Theorem~\ref{thm:count} yields at least $2^{2^{k-6}}$ distinct Hamilton cycles in~$G_k^+$ with the properties discussed above, which give rise to at least $\smash[t]{2^{2^{k-6}}}$ distinct Hamilton cycles in~$H(2k+1,k)$ by the construction described above.
\end{proof}

\section{Open problems}
\label{sec:open}

As mentioned in Section~\ref{sec:algo}, our construction of a Hamilton cycle in the odd graph~$O_k$ translates straightforwardly into an algorithm to compute this cycle in polynomial time, that is, polynomial in the size of the graph.
It remains open whether this can be improved to an algorithm whose running time is polynomial in~$k$, ideally, constant for each generated vertex.
This might even yield a simpler algorithmic solution for the middle levels conjecture (recall Section~\ref{sec:mlc} and \cite{MR4075363}).

As Hamiltonicity of the odd graphs~$O_k$ is now settled, one may turn to Biggs' more general conjecture that~$O_k$ has $\lfloor(k+1)/2\rfloor$ edge-disjoint Hamilton cycles for $k\geq 3$~\cite{MR556008}.
Clearly, edge-disjointness is a much stronger requirement than distinctness guaranteed by Theorem~\ref{thm:count}.
A starting point for such a construction may be the factorizations of~$O_k$ into cycles described by Johnson and Kierstead~\cite{MR2128031}.
Our construction of a cycle factor is unrelated to these factorizations, and we do not even see how to extend it to a construction of two edge-disjoint cycle factors.

Another intriguing open problem is whether all Kneser graphs~$K(n,k)$ except the Petersen graph $O_2=K(5,2)$ are Hamiltonian.
Despite serious attempts, we have not been able to generalize the methods presented in this paper, nor the inductive decomposition technique of Johnson~\cite{MR2046083}, to the general case.
In light of Theorem~\ref{thm:sparse}, the sparsest open case is $n=2k+3$.

A strengthening of the concept of containing a Hamilton cycle is to contain the $r$th power of a Hamilton cycle.
To this end, Katona~\cite{MR2181045} conjectured that the vertices of $K(n,k)$ can be ordered so that any $r+1$ consecutive vertices for $r:=\lfloor n/k\rfloor-2$ are disjoint sets, which he proved for $k=2$ using Walecki's theorem.
It seems plausible that Katona's conjecture holds even for $r:=\lceil n/k\rceil-2$.
Theorem~\ref{thm:odd} confirms this for the case~$n=2k+1$ (where $r=1$).

\bibliographystyle{alpha}
\bibliography{odd-refs}

\end{document}